\documentclass{article}

\usepackage[utf8]{inputenc} % allow utf-8 input
\usepackage[T1]{fontenc}    % use 8-bit T1 fonts
\usepackage{hyperref}      % hyperlinks
\usepackage{url}            			% simple URL typesetting
\usepackage{array}
\usepackage{amsthm}
\usepackage{booktabs}       		% professional-quality tables
\usepackage{amsfonts}       		% blackboard math symbols
\usepackage{amsmath}
\usepackage{dsfont}

\usepackage{pifont}	
\usepackage{nicefrac}       		% compact symbols for 1/2, etc.
\usepackage{microtype}      		% microtypography
\usepackage{lipsum}				% Can be removed after putting your text content
\usepackage[square,numbers]{natbib}
\usepackage{mathtools}
\usepackage{algorithm}
\usepackage{algorithmicx}
\usepackage{algpseudocode}
\usepackage{graphicx}
\usepackage{indentfirst,latexsym,bm}
\usepackage{subfigure}
\usepackage{paralist}
\usepackage{wrapfig}
\usepackage{caption}
\usepackage{amssymb}
\usepackage{xcolor}
\usepackage{comment}
\usepackage{enumitem}
\usepackage{bbm}
\usepackage{tikz}
\usetikzlibrary{arrows.meta,positioning}
\usepackage{pgfplots}
\pgfplotsset{compat=newest}
\usepgfplotslibrary{fillbetween}
\usetikzlibrary{shapes,decorations}
\usetikzlibrary{fit}

\usepackage{varwidth}
\usepackage{multirow}

\usepackage{framed}

\colorlet{color1}{blue}
\colorlet{color2}{red!50!black}

\definecolor{ivory}{RGB}{218,215,203}

\definecolor{cuhkp}{RGB}{98,56,105} 	% purple dark
\definecolor{cuhkpl}{RGB}{152,24,147} 	% purple light
\definecolor{cuhkb}{RGB}{219,160,1} 	% ocher
\definecolor{cuhkbd}{RGB}{178,129,0} 	% ocher dark
\definecolor{cuhkr}{RGB}{88,35,155}  	% magenta-red
\definecolor{mydarkblue}{rgb}{0,0.08,0.45} %dark blue

\hypersetup{
  colorlinks=true,
  frenchlinks=false,
  pdfborder={0 0 0},
  naturalnames=false,
  hypertexnames=false,
  breaklinks,
  linkcolor=mydarkblue,
  citecolor=mydarkblue,
  filecolor=mydarkblue,
  urlcolor = black,
}

\RequirePackage[capitalize,nameinlink]{cleveref}[0.19]

\crefname{section}{section}{sections}
\crefname{subsection}{subsection}{subsections}

\Crefname{figure}{Figure}{Figures}

\crefformat{equation}{\textup{#2(#1)#3}}
\crefrangeformat{equation}{\textup{#3(#1)#4--#5(#2)#6}}
\crefmultiformat{equation}{\textup{#2(#1)#3}}{ and \textup{#2(#1)#3}}
{, \textup{#2(#1)#3}}{, and \textup{#2(#1)#3}}
\crefrangemultiformat{equation}{\textup{#3(#1)#4--#5(#2)#6}}%
{ and \textup{#3(#1)#4--#5(#2)#6}}{, \textup{#3(#1)#4--#5(#2)#6}}{, and \textup{#3(#1)#4--#5(#2)#6}}

\Crefformat{equation}{#2Equation~\textup{(#1)}#3}
\Crefrangeformat{equation}{Equations~\textup{#3(#1)#4--#5(#2)#6}}
\Crefmultiformat{equation}{Equations~\textup{#2(#1)#3}}{ and \textup{#2(#1)#3}}
{, \textup{#2(#1)#3}}{, and \textup{#2(#1)#3}}
\Crefrangemultiformat{equation}{Equations~\textup{#3(#1)#4--#5(#2)#6}}%
{ and \textup{#3(#1)#4--#5(#2)#6}}{, \textup{#3(#1)#4--#5(#2)#6}}{, and \textup{#3(#1)#4--#5(#2)#6}}

\crefdefaultlabelformat{#2\textup{#1}#3}

\theoremstyle{plain}

\newtheorem{thm}{Theorem}[section]
\newtheorem{lemma}[thm]{Lemma}
\newtheorem{corollary}[thm]{Corollary}

\theoremstyle{plain}
\newtheorem{defn}[thm]{Definition}

\theoremstyle{definition}
\newtheorem{remark}[thm]{Remark}

\newcolumntype{C}[1]{>{\centering\arraybackslash}p{#1}}

\newcommand{\e}{\begin{equation}}
\newcommand{\ee}{\end{equation}}
\newcommand{\bmat}{\begin{bmatrix}}
\newcommand{\emat}{\end{bmatrix}}
\newcommand{\btab}{\begin{tabular}}
\newcommand{\etab}{\end{tabular}}
\newcommand{\Prob}{\mathbb{P}}
\newcommand{\Exp}{\mathbb{E}}

\newcommand{\vct}[1]{\boldsymbol{#1}}
\newcommand{\mtx}[1]{\boldsymbol{#1}}
\DeclareMathOperator*{\argmin}{\text{argmin}}

\newcommand{\prox}{\mathrm{prox}}
\newcommand{\env}{\mathrm{env}}

\newcommand{\dom}{\operatorname{dom}}

\newcommand{\R}{\mathbb{R}}
\newcommand{\Rn}{\mathbb{R}^n}

\newcommand{\N}{\mathbb{N}}

\newcommand{\Rex}{(-\infty,\infty]}
\newcommand{\iprod}[2]{\langle #1, #2 \rangle}

\newcommand{\calF}{\mathcal{F}}

\newcommand{\calO}{\mathcal{O}}

\newcommand{\RR}{{\mathsf{RR}}}
\newcommand{\PSGD}{\mathsf{prox}\text{-}\mathsf{SGD}}
\newcommand{\SGD}{\mathsf{SGD}}
\newcommand{\SMM}{\mathsf{SMM}}

\newcommand{\vg}{\vct{g}}

\newcommand{\vp}{\vct{p}}

\newcommand{\vs}{\vct{s}}

\newcommand{\vv}{\vct{v}}

\newcommand{\vx}{\vct{x}}
\newcommand{\vy}{\vct{y}}

\newcommand{\vPhi}{\vct{\Phi}}
\newcommand{\bA}{\vct{A}}
\newcommand{\bB}{\vct{B}}
\newcommand{\bM}{\vct{M}}
\newcommand{\bE}{\vct{E}}

\newcommand{\mA}{\mtx{A}}
\newcommand{\mB}{\mtx{B}}

\definecolor{redp}{RGB}{255,0,0}
\definecolor{greenp}{RGB}{128,230,0}

\newcommand{\cmark}{\color{greenp}\checkmark}%
\newcommand{\xmark}{\color{redp}\ding{55}}

\usepackage[preprint]{neurips_2022}

\title{A Unified Convergence Theorem for Stochastic Optimization Methods}

\author{%
  Xiao Li\\
  School of Data Science (SDS)\\
  The Chinese University of Hong Kong, Shenzhen\\
  Shenzhen, China \\
  \texttt{lixiao@cuhk.edu.cn} 
   \And
   Andre Milzarek \\
   School of Data Science (SDS) \\
   The Chinese University of Hong Kong, Shenzhen \\
   Shenzhen Research Institute of Big Data (SRIBD) \\
   Shenzhen, China \\
   \texttt{andremilzarek@cuhk.edu.cn} 
}

\AtBeginDocument{\setlength\abovedisplayskip{5pt}}
\AtBeginDocument{\setlength\belowdisplayskip{5pt}}

\begin{document}

\maketitle

\begin{abstract}
  In this work, we provide a fundamental unified convergence theorem used for deriving expected and almost sure convergence results for a series of stochastic optimization methods. Our unified theorem only requires to verify several representative conditions and is not tailored to any specific algorithm. As a direct application, we recover expected and almost sure convergence results of the stochastic gradient method ($\SGD$) and random reshuffling ($\RR$) under more general settings. Moreover, we establish new expected and almost sure convergence results for the stochastic proximal gradient method ($\PSGD$) and stochastic model-based methods for nonsmooth nonconvex optimization problems. These applications reveal that our unified theorem provides a plugin-type convergence analysis and strong convergence guarantees for a wide class of stochastic optimization methods. 
\end{abstract}

\section{Introduction}\label{sec:intro}
Stochastic optimization methods are widely used to solve stochastic optimization problems and empirical risk minimization, serving as one of the foundations of machine learning. Among the many different stochastic methods, the most classic one is the stochastic gradient method ($\SGD$), which dates back to Robbins and Monro  \cite{robbins1951}. If the problem at hand has a finite-sum structure, then another popular stochastic method is  random reshuffling ($\RR$) \cite{gurbu2019}.    When the objective function has a composite form or is weakly convex (nonsmooth and nonconvex), then the stochastic proximal gradient method ($\PSGD$) and stochastic model-based algorithms are the most typical approaches \cite{GhaLanZha16,DavDru19}.
Apart from the mentioned stochastic methods, there are many others like $\SGD$ with momentum, Adam, stochastic higher order methods, etc. In this work, our goal is to establish and understand fundamental \emph{convergence} properties of these stochastic optimization methods via a novel unified convergence framework. 

\textbf{Motivations.} 
Suppose we apply $\SGD$ to minimize a smooth nonconvex function $f$. $\SGD$ generates a sequence of iterates $\{\vx^k\}_{k\geq 0}$, which is a stochastic process due to the randomness of the algorithm and the utilized stochastic oracles. The most commonly seen `convergence result' for $\SGD$ is the \emph{expected iteration complexity}, which typically takes the form \cite{ghadimi2013}
%
%\e \min\limits_{k=0,\ldots, T} \ \Exp[\|\nabla f(\vx^k)\|^2] \leq \calO(({T+1})^{-\frac12}) \quad \text{or} \quad \Exp[\|\nabla f(\vx^{\bar k})\|^2] \leq \calO(({T+1})^{-\frac12}), \label{eq:complexity} \ee
\e \min_{k=0,\ldots, T} \ \Exp[\|\nabla f(\vx^k)\|^2] \leq \calO\left(\frac{1}{\sqrt{T+1}}\right) \quad \text{or} \quad \Exp[\|\nabla f(\vx^{\bar k})\|^2] \leq \calO\left(\frac{1}{\sqrt{T+1}}\right), \label{eq:complexity} \ee
where $T$ denotes the total number of iterations and $\bar k$ is an index sampled uniformly at random from $\{0,\ldots, T\}$.  Note that we ignored some higher-order convergence terms and constants to ease the presentation. Complexity results are integral to understand core properties and progress of the algorithm during the first $T$ iterations, while the asymptotic convergence behavior plays an equally important role as it  characterizes whether an algorithm can eventually approach an exact stationary point or not.  We refer to \Cref{appen:motivations} for additional motivational background for studying asymptotic convergence properties of stochastic optimization methods. Here, an \emph{expected convergence result}, associated with the nonconvex minimization problem $\min_{\vx} f(\vx)$,  has the form 
\e
\lim_{k \to \infty} \Exp[\|\nabla f(\vx^{ k})\|] = 0. \label{eq:hello-expectation}
\ee
Intuitively, it should be possible to derive expected convergence from the expected iteration complexity \eqref{eq:complexity} by letting $T\to \infty$. However, this is not the case as the `$\min$' operator and the sampled $\bar k$ are not well defined or become meaningless when $T$ goes to $\infty$.

%Another observation is that 
The above results are stated in expectation and describe the behavior of the algorithm by averaging  infinitely many runs. Though this is an important convergence measure, in practical situations the algorithm is often only run once and the last iterate is returned as a solution. This observation motivates and necessitates \emph{almost sure convergence results}, which establish convergence with probability $1$ for a single run of the stochastic method:
\e
\lim_{k \to \infty} \|\nabla f(\vx^{ k})\| = 0 \quad \text{almost surely}. \label{eq:hello-almost-surely}
\ee
\vspace{-0.4cm}

\textbf{Backgrounds.}  Expected and almost sure convergence results have been extensively studied for convex optimization; see, e.g., \cite{culioli1990,polyak1992,shamir2013,zhang2004,bottou2003,sebbouh2021}. Almost sure convergence of $\SGD$ for minimizing a smooth nonconvex function $f$ was provided in the seminal work \cite{BerTsi00} using very standard assumptions, i.e., Lipschitz continuous $\nabla f$ and bounded variance.   Under the same conditions,  the same almost sure convergence of $\SGD$ was established in  \cite{Orabona20} based on a much simpler argument than that of \cite{BerTsi00}. A weaker `$\liminf$'-type almost sure convergence result for $\SGD$ with AdaGrad step sizes was shown in \cite{li2019con}.   {Recently, the work \cite{mertikopoulos2020} derives almost sure convergence of $\SGD$ under the assumptions that $f$ and $\nabla f$  are Lipschitz continuous, $f$ is coercive, $f$ is not asymptotically flat, and the $\upsilon$-th moment of the stochastic error is bounded with $\upsilon\geq 2$. This result relies on stronger assumptions than the base results in \cite{BerTsi00}.  Nonetheless, it allows more aggressive diminishing step sizes if $\upsilon>2$.}
%Nonetheless,  we have to clarify for \cite{mertikopoulos2020}  that their main objective  is to provide almost surely strict saddle point avoidance of $\SGD$ and hence, requires these strong assumptions. 
 {Apart from standard $\SGD$, almost sure convergence of different respective variants for min-max problems was discussed in \cite{hsieh2021limits}.}
In terms of expected convergence, the work \cite{BotCurNoc18} showed $\lim_{k \to \infty} \Exp[\|\nabla f(\vx^{ k})\|] = 0$ under the additional assumptions that $f$ is twice continuously differentiable and the multiplication of the Hessian and gradient $\nabla ^2f (\vx) \nabla f(\vx)$ is Lipschitz continuous. 

Though the convergence of $\SGD$ is well-understood and a classical topic, asymptotic convergence results of the type \eqref{eq:hello-expectation} and \eqref{eq:hello-almost-surely} often require a careful and separate analysis for other stochastic optimization methods --- especially when the objective function is simultaneously nonsmooth and nonconvex. In fact and as outlined, a direct transition from the more common complexity results \eqref{eq:complexity} to the full convergence results \eqref{eq:hello-expectation} and \eqref{eq:hello-almost-surely} is often not possible without further investigation.  %see our later detailed literature review in \Cref{sec:new-results}. 

\textbf{Main contributions.} 
We provide a fundamental \emph{unified convergence theorem} (see \Cref{thm:convergence theorem}) for deriving both expected and almost sure convergence of stochastic optimization methods. Our theorem is not tailored to any specific algorithm, instead it incorporates several abstract conditions that suit a vast and general class of problem structures and algorithms. The proof of this theorem is elementary.  

We then apply our novel theoretical framework to several classical stochastic optimization methods to recover existing and to establish new convergence results.
Specifically, we recover expected and almost sure convergence results for $\SGD$ and $\RR$. Though these results are largely known in the literature, we derive unified and slightly stronger results under a general ABC condition \cite{lei2019stochastic,khaled2020better} rather than the standard  bounded variance assumption. We also remove the stringent assumption used in  \cite{BotCurNoc18} to show \eqref{eq:hello-expectation} for $\SGD$. 
As a core application of our framework, we derive expected and almost sure convergence results for $\PSGD$ in the nonconvex setting and under the more general ABC condition and for stochastic model-based methods under very standard assumptions. In particular, we show that the iterates $\{\vx^k\}_{k \geq 0}$ generated by $\PSGD$ and other stochastic model-based methods will approach the set of stationary points almost surely and in an expectation sense. These results are \emph{new} to our knowledge (see also \Cref{sec:literature} for further discussion).

The above applications illustrate the general plugin-type purpose of our unified convergence analysis framework. Based on the given recursion and certain properties of the algorithmic update, we can derive broad convergence results by utilizing our theorem, which can significantly simplify the convergence analysis of stochastic optimization methods; see \Cref{sec:steps} for a summary.

\section{A unified convergence theorem}
%We first introduce the underlying probability space. 
Throughout this work, let $(\Omega,\mathcal F,\{\mathcal F_k\}_{k\geq 0},\Prob)$ be a filtered probability space and let us assume that the sequence of iterates $\{\vx^k\}_{k\geq 0}$ is adapted to the filtration $\{\mathcal F_k\}_{k\geq 0}$, i.e., each of the random vectors $\vx^k : \Omega \to \Rn$ is $\mathcal F_k$-measurable.

In this section, we present a unified convergence theorem for the sequence $\{\vx^k\}_{k \geq 0}$ based on an abstract convergence measure $\vPhi$. {To make the abstract convergence theorem more accessible, the readers may momentarily regard $\vPhi$ and $\{\mu_k\}_{k\geq 0}$ as $\nabla f$ and the sequence related to the step sizes, respectively.} We then present the main steps for showing the convergence of a stochastic optimization method by following a step-by-step verification of the conditions in  our unified convergence theorem.

\begin{thm}\label{thm:convergence theorem}
	 Let the mapping $\vPhi: \R^n \rightarrow \R^m$ and the sequences $\{\vx^k\}_{k\geq 0}\subseteq \R^n$ and $\{\mu_k\}_{k\geq 0} \subseteq \R_{++}$ be given. Consider the following conditions:    
	\begin{enumerate}[label=\textup{\textrm{(P.\arabic*)}},topsep=0pt,itemsep=0ex,partopsep=0ex]
		\item \label{P1} The function $\vPhi$ is ${\sf L}_{\Phi}$-Lipschitz continuous for some ${\sf L}_{\Phi}>0$, i.e., we have $\|\vPhi(\vx) - \vPhi(\vy)\| \leq  {\sf L}_{\Phi} \|\vx-\vy\|$ for all $\vx,\vy \in \Rn$. 
		\item \label{P2} There exists a constant $a > 0$ such that $\sum_{k=0}^{\infty} \ \mu_k \, \Exp[\|\vPhi(\vx^k) \|^a]  < \infty$.
	\end{enumerate}
	The following statements are valid: 
	\begin{enumerate}[label=\textup{\textrm{(\roman*)}},topsep=0pt,itemsep=0ex,partopsep=0ex]
		\item Let the conditions \ref{P1}--\ref{P2} be satisfied and suppose further that
		\begin{enumerate}[label=\textup{\textrm{(P.\arabic*)}},topsep=0pt,itemsep=0ex,partopsep=0ex,start=3]
		\item \label{P3} There exist  constants ${\sf A}, {\sf B}, b \geq 0$ and $ p_1, p_2, q >0$ such that 
		\[ \Exp[\|\vx^{k+1} - \vx^k\|^q] \leq  {\sf A} \mu_k^{p_1} \ + \ {\sf B} \mu_k^{p_2}  \, \Exp[\|\vPhi(\vx^k) \|^b]. \]	 
		 \item \label{P4} The sequence $\{\mu_k\}_{k\geq 0}$  and the  parameters $a, b, q, p_1, p_2$ satisfy
		 \[
		  \{\mu_k\}_{k\geq 0} \; \text{is bounded}, \quad {\sum}_{k=0}^{\infty} \ \mu_k = \infty, \quad \text{and} \quad a, q \geq 1, \;\; a \geq b,  \;\; p_1, p_2 \geq q. 
		  \]
    	\end{enumerate}
        Then, it holds that $\lim_{k \to \infty} \Exp[\|\vPhi(\vx^k)\|] = 0$. 
		\item Let the properties \ref{P1}--\ref{P2} hold and assume further that 
		\begin{enumerate}[label=\textup{\textrm{(P.\arabic*${}^\prime$)}},topsep=0pt,itemsep=0ex,partopsep=0ex,start=3]
			\item \label{P3'} There exist constants ${\sf A}, b \geq 0$, $p_1, p_2, q > 0$ and random vectors $\bA_k, \bB_k : \Omega \to \R^n$ such that
			\[ \vx^{k+1} = \vx^k + \mu_k^{p_1} \bA_k + \mu_k^{p_2} \bB_k \]
			and for all $k$,  $\bA_k, \bB_k$ are $\mathcal F_{k+1}$-measurable and we have $\Exp[\bA_k \mid \mathcal F_k] = 0$ almost surely, $\Exp[\|\bA_k\|^q] \leq {\sf A}$, and $\limsup_{k \to \infty} \|\bB_k\|^q /(1+\|\vPhi(\vx^k)\|^b) < \infty$ almost surely.
			\item \label{P4'} The sequence $\{\mu_k\}_{k\geq 0}$  and the  parameters $a, b, q, p_1, p_2$ satisfy $\mu_k \to 0$,
			\[
			{\sum}_{k=0}^{\infty} \ \mu_k = \infty, \quad {\sum}_{k=0}^\infty \, \mu_k^{2p_1} < \infty,   \quad  \text{and} \quad q \geq 2, \;\; qa \geq b, \ \ p_1 > \frac12, \;\; p_2 \geq 1.
			\]
		\end{enumerate}
		Then, it holds that $\lim_{k \to \infty} \|\vPhi(\vx^k)\| = 0$ almost surely. 
	\end{enumerate}
\end{thm}

The proof of \Cref{thm:convergence theorem} is elementary. We provide the core ideas here and defer its proof to \Cref{appen:proof main thm}.  Item (i) is proved by contradiction. An easy first result is $\liminf_{k\to\infty} \Exp[\|\vPhi(\vx^k)\|^a] = 0$.  We proceed and assume that $\{\Exp[\|\vPhi(\vx^k)\|]\}_{k\geq 0}$ does not converge to zero. Then, for some  $\delta > 0$, we can construct  two subsequences $\{\ell_t\}_{t \geq 0}$ and $\{u_t\}_{t \geq 0}$ such that $\ell_t < u_t$ and
$\Exp[\|\vPhi(\vx^{\ell_t})\|] \geq 2\delta$,  $\Exp[\|\vPhi(\vx^{u_t})\|^a] \leq \delta^a$, and $\Exp[\|\vPhi(\vx^k)\|^a] > \delta^a$ for all  $\ell_t < k < u_t$. Based on this construction, the conditions in the theorem, and a set of inequalities, we will eventually reach a contradiction. We notice that the Lipschitz continuity of $\vPhi$ plays a prominent role when establishing this contradiction. Our overall proof  strategy is inspired by the analysis of classical trust region-type methods,
see, e.g., \cite[Theorem 6.4.6]{ConGouToi00}.
%
%By assuming $\{\Exp[\|\vPhi(\vx^k)\|]\}_{k\geq 0}$ does not converge to zero. Then, for some  $\delta > 0$, we can construct  two subsequences $\{\ell_t\}_{t \geq 0}$ and $\{u_t\}_{t \geq 0}$ such that $\ell_t < u_t$ and
%$\Exp[\|\vPhi(\vx^{\ell_t})\|] \geq 2\delta$,  $\Exp[\|\vPhi(\vx^{u_t})\|^a] \leq \delta^a$, and $\Exp[\|\vPhi(\vx^k)\|^a] > \delta^a$ for all  $\ell_t < k < u_t$. Based on this construction, the conditions in the theorem, and a set of inequalities, we can show $\Exp[\|\vx^{u_t} - \vx^{\ell_t}\|]  \to 0$ as $t \to \infty$. This is a contradiction to $\Exp[\|\vx^{u_t} - \vx^{\ell}\|]\geq \delta/{\sf L}_{\Phi}$ for all $t\geq 0$, which follows from \ref{P1} and the above construction.  
For item (ii), %in order to establish the almost sure convergence result, 
we first control the stochastic behavior of the error terms $\bA_k$ by martingale convergence theory. We can then conduct sample-based arguments to derive the final result, which is essentially deterministic and hence, %allows us to 
follows similar arguments to that of item (i).   

{The major application areas of our unified convergence framework comprise stochastic optimization methods that have non-vanishing stochastic errors or that utilize diminishing step sizes.} In the next subsection, we state the main steps for showing convergence of stochastic optimization methods.  This also clarifies the abstract conditions listed in the theorem.

 \subsection{The steps for showing convergence of stochastic optimization methods}\label{sec:steps}
 In order to apply the unified convergence theorem, we have to verify the conditions stated in the theorem, resulting in three main phases below. 

\textbf{Phase I: Verifying \ref{P1}--\ref{P2}.}  Conditions \ref{P1}--\ref{P2} are used for both the expected and the almost sure convergence results.    Condition \ref{P1} is a problem property and is very standard. We present the final convergence results in terms of the abstract measure $\vPhi$. This measure can be regarded as $f - f^*$ in convex optimization, {$\nabla f$} in smooth nonconvex optimization,  {the gradient} of the Moreau envelope in weakly convex optimization, etc. In all the situations, assuming Lipschitz continuity of the convergence measure $\vPhi$ is standard and is arguably a minimal assumption in order to obtain  iteration complexity and/or convergence results.   

Condition \ref{P2} is typically a result of the algorithmic property or complexity analysis. To verify this condition, one first establishes the recursion of the stochastic method, which almost always has the form
\[
\Exp[\vy_{k+1} \mid \mathcal F_k]\leq (1+\beta_k) \vy_k - \mu_k \|\vPhi(\vx^k)\|^a + \zeta_k. 
\] 
Here, $\vy_k$ is a suitable Lyapunov function measuring the (approximate) descent property of the stochastic method, $\zeta_k$ represents the error term satisfying $\sum_{k=0}^{\infty} \zeta_k <\infty$, $\beta_k$ is often related to the step sizes and satisfies $\sum_{k=0}^\infty \beta_k < \infty$. Then, applying the supermartingale convergence theorem (see \Cref{Theorem:martingale_convergece}), we obtain $\sum_{k=0}^{\infty} \ \mu_k\, \Exp[\|\vPhi(\vx^k) \|^a]  < \infty$, i.e., condition \ref{P2}. 

{Since condition \ref{P2} is typically a consequence of the underlying algorithmic recursion, one can also derive the standard finite-time complexity bound \eqref{eq:complexity} in terms of the measure $ \Exp[\|\vPhi(\vx^k) \|^a]$ based on it. Hence, non-asymptotic complexity results  are also included implicitly in our framework as a special case. To be more specific, \ref{P2} implies $\sum_{k=0}^{T} \mu_k \Exp[\|\vPhi(\vx^k) \|^a]\leq M$ for some constant $M>0$ and some total number of iterations $T$. This then yields $ \min_{0\leq k\leq T} \Exp[\|\vPhi(\vx^k) \|^a] \leq M / \sum_{k=0}^{T} \mu_k $. Note that the sequence $\{\mu_k\}_{k\geq 0}$ is often related to the step sizes. Thus, choosing the step sizes properly results in the standard finite-time complexity result.}

\textbf{Phase II: Verifying  \ref{P3}--\ref{P4} for showing expected convergence.} Condition \ref{P3} requires an upper bound on the step length of the update in terms of expectation, including upper bounds for the search direction  and the stochastic error of the algorithm. It is often related to certain bounded variance-type assumptions for analyzing stochastic methods. For instance, \ref{P3} is satisfied under the standard bounded variance assumption for $\SGD$, the more general ABC assumption for $\SGD$, the bounded stochastic subgradients assumption, etc. Condition \ref{P4} is a standard diminishing step sizes condition used in stochastic optimization. % once $\{\mu_k\}_{k\geq 0}$ is certain step sizes quantity. 

Then, one can apply item (i) of \Cref{thm:convergence theorem} to obtain $\Exp[\|\vPhi(\vx^k)\|] \to 0$.

\textbf{Phase III: Verifying  \ref{P3'}--\ref{P4'} for showing almost sure convergence.} Condition \ref{P3'}  is parallel to \ref{P3}.  It decomposes the update into a martingale term $\mA_k$ and a bounded error term $\mB_k$. We will see later that this condition holds true for many stochastic methods.  Though this condition requires the update to have a certain decomposable form, it indeed can be verified by bounding the step length of the update in conditional expectation, which is similar to \ref{P3}. Hence, \ref{P3'} can be interpreted as a conditional version of \ref{P3}. To see this, we can construct 
\e\label{eq:decom update}
    \vx^{k+1} = \vx^k  + \mu_{k} \cdot \underbracket{\begin{minipage}[t][3ex][t]{36ex}\centering$\frac{1}{\mu_k} \left(\vx^{k+1}-\vx^k -\Exp[\vx^{k+1} - \vx^k \mid \calF_k] \right)$\end{minipage}}_{\mA_k} + \mu_{k} \cdot \underbracket{\begin{minipage}[t][3ex][t]{21ex}\centering$\frac{1}{\mu_k} \Exp[\vx^{k+1} - \vx^k \mid \calF_k]$\end{minipage}}_{\mB_k}. 
\ee
By Jensen's inequality, we then have $\Exp[\mA_k \mid \calF_k] = 0$,
\[
 \Exp[\|\mA_k\|^q] \leq {2^{q}}{\mu_k^{-q}} \cdot \Exp[\|\vx^{k+1} - \vx^k\|^q], \quad \text{and} \quad \|\mB_k\|^q \leq {\mu_k^{-q}} \cdot \Exp[\|\vx^{k+1} - \vx^k\|^q\mid \calF_k]. 
\]
Thus, once it is possible to derive $\Exp[\|\vx^{k+1} - \vx^k\|^q\mid \calF_k] = \calO(\mu_k^q)$ in an almost sure sense, condition \ref{P3'} is verified with $p_1 = p_2 = 1$. Condition \ref{P4'}  is parallel to \ref{P4} and is standard in stochastic optimization. Application of item (ii) of \Cref{thm:convergence theorem} then yields $\|\vPhi(\vx^k)\| \to 0$ almost surely.

 In the next section, we will illustrate how to show convergence for a set of classic stochastic methods by following the above three steps.

\section{Applications to stochastic optimization methods}\label{sec:app}
%In this section, we follow the steps  in \Cref{sec:steps} to recover convergence results of $\SGD$ for smooth nonconvex optimization problems. 

\subsection{Convergence results of SGD}\label{sec:SGD}
We consider the standard $\SGD$ method for solving the smooth optimization problem $\min_{\vx \in \Rn}\,f(\vx)$, where the iteration of $\SGD$ is given by 
\e\label{eq:SGD}
\vx^{k+1} = \vx^k - \alpha_k \vg^k. 
\ee
Here, $\vg^k$ denotes a stochastic approximation of the gradient $\nabla f(\vx^k)$. We assume that each stochastic gradient $\vg^k$ is $\mathcal F_{k+1}$-measurable and that the generated stochastic process $\{\vx^k\}_{k \geq 0}$ is adapted to the filtration $\{\mathcal F_k\}_{k \geq 0}$. We consider the following standard assumptions:
\begin{enumerate}[label=\textup{\textrm{(A.\arabic*)}},topsep=0pt,itemsep=0ex,partopsep=0ex]
	\item \label{A1} The mapping $\nabla f : \Rn \to \Rn$ is Lipschitz continuous on $\Rn$ with modulus ${\sf L} > 0$.
	\item \label{A2} The objective function $f$ is bounded from below on $\Rn$, i.e., there is $\bar f$ such that $f(\vx) \geq \bar f$ for all $\vx \in \Rn$.
	\item \label{A3} Each oracle $\vg^k$ defines an unbiased estimator of $\nabla f(\vx^k)$, i.e., it holds that $\Exp[\vg^k\mid \mathcal F_k] = \nabla f(\vx^k)$ almost surely, and  there exist ${\sf C}, {\sf D} \geq 0$ such that
	\[\Exp[\|\vg^k - \nabla f(\vx^k)\|^2 \mid \mathcal F_k] \leq {\sf C} [f(\vx^k)-\bar f] + {\sf D} \quad \text{almost surely} \quad  \forall~k \in \N.  \]
	\item \label{A4} The step sizes $\{\alpha_k\}_{k \geq 0}$ satisfy $\sum_{k=0}^\infty \alpha_k = \infty$ and $\sum_{k=0}^\infty \alpha_k^2 < \infty$.
\end{enumerate}
We now derive the convergence of $\SGD$ below by setting $\vPhi \equiv \nabla f$ and $\mu_k \equiv \alpha_k$. 

\textbf{Phase I: Verifying  \ref{P1}--\ref{P2}.}
\ref{A1} verifies condition \ref{P1} with ${\sf L}_\Phi \equiv {\sf L}$.  We now check \ref{P2}. Using \ref{A2},  \ref{A3},  and a standard analysis for $\SGD$ gives the following recursion (see \Cref{appen:SGD} for the full derivation): 
\e \label{eq:SGD recursion}
 \Exp[f(\vx^{k+1}) - \bar f \mid \mathcal F_k] \leq \left(1 + \frac{{\sf LC}\alpha_k^2}{2}\right)[f(\vx^k) - \bar f] - \alpha_k \left(1-\frac{{\sf L}\alpha_k}{2}\right) \|\nabla f(\vx^k)\|^2 + \frac{{\sf LD}\alpha_k^2}{2}. 
\ee
Taking total expectation, using \ref{A4}, and applying the supermartingale convergence theorem (\Cref{Theorem:martingale_convergece}) gives $\sum_{k=0}^\infty \alpha_k \Exp[\|\nabla f(\vx^k)\|^2]<\infty$. Furthermore, the sequence $\{\Exp[f(\vx^{k})]\}_{k \geq 0}$ converges to some finite value. This verifies \ref{P2} with $a = 2$. 

\textbf{Phase II: Verifying  \ref{P3}--\ref{P4} for showing expected convergence.}
 For \ref{P3}, we have by \eqref{eq:SGD} and \ref{A3} that 
\[\Exp[\|\vx^{k+1}-\vx^k\|^2] \leq \alpha_k^2 \Exp[ \|\nabla f(\vx^k)\|^2] + {\sf C}\alpha_k^2\Exp[f(\vx^k)-\bar f] + {\sf D}\alpha_k^2. \] 
Due to the convergence of $\{\Exp[f(\vx^{k})]\}_{k \geq 0}$, there exists ${\sf F}$ such that $\Exp[f(\vx^k)-\bar f] \leq {\sf F}$ for all $k$.
Thus, condition \ref{P3} holds with $q=2$, ${\sf A} =  {\sf C}{\sf F} + {\sf D}$, $p_1 =2$, ${\sf B} = 1$, $p_2 = 2$, and $b = 2$.
Condition \ref{P4} is verified by \ref{A4} and the previous parameters choices.  Therefore, we can apply \Cref{thm:convergence theorem} to deduce $\Exp[\|\nabla f(\vx^k)\|] \to 0$. 

\textbf{Phase III: Verifying  \ref{P3'}--\ref{P4'} for showing almost sure convergence.}
 For \ref{P3'}, it follows from the update \eqref{eq:SGD} that 
\[
     \vx^{k+1} = \vx^k - \alpha_k (\vg^k - \nabla f(\vx^k)) - \alpha_k \nabla f(\vx^k). 
\]
We have $p_1 = 1$, $\mA_k = \vg^k - \nabla f(\vx^k)$, $p_2 = 1$, and $\mB_k = \nabla f(\vx^k)$. Using \ref{A2},  \ref{A3},  $\Exp[f(\vx^{k}) - \bar f] \leq {\sf F}$, and choosing any $q = b>0$ establishes \ref{P3'}. As before, condition \ref{P4'} follows from \ref{A4} and the previous parameters choices. Applying \cref{thm:convergence theorem} yields $\|\nabla f(\vx^k)\| \to 0$ almost surely. 

Finally, we summarize the above results in the following corollary.
\begin{corollary}\label{thm:SGD}
 Let us consider $\SGD$ \eqref{eq:SGD} for smooth nonconvex optimization problems under \ref{A1}--\ref{A4}. Then, we have $\lim_{k \to \infty} \Exp[\|\nabla f(\vx^k)\|] = 0$ and $\lim_{k \to \infty} \|\nabla f(\vx^k)\| = 0$ almost surely. 
\end{corollary}

\subsection{Convergence results of random reshuffling}\label{sec:RR}

We now consider random reshuffling ($\RR$) applied to problems with a finite sum structure
\[ \min_{\vx \in \Rn} f(\vx) := \frac{1}{N}{\sum}_{i=1}^N f(\vx,i), \]
where each component function $f(\cdot,i) : \Rn \to \R$ is supposed to be smooth. At iteration $k$, $\RR$ first generates a random permutation $\sigma^{k+1}$ of the index set $\{1,\dots,N\}$. It then updates $\vx^{k}$ to ${\vx^{k+1}}$ through $N$ consecutive gradient descent-type steps by accessing and using the component gradients $\{\nabla f(\cdot,\sigma_1^{k+1}),\dots,\nabla f(\cdot,\sigma_N^{k+1})\}$ sequentially. Specifically, one update-loop (epoch) of $\RR$ is given by
\e\label{eq:RR}
\tilde \vx_0^k = \vx^k, \quad \tilde \vx_i^k = \tilde \vx_{i-1}^k - \alpha_k \nabla f(\tilde \vx_{i-1}^k,\sigma_i^{k+1}), \quad i = 1,\dots,N, \quad \vx^{k+1} = \tilde \vx_N^k. 
\ee
After one such loop, the step size $\alpha_k$ and the permutation $\sigma^{k+1}$ is updated accordingly; cf. \cite{gurbu2019,mishchenko2020,nguyen2020unified}. We make the following standard assumptions: 

\begin{enumerate}[label=\textup{\textrm{(B.\arabic*)}},topsep=0pt,itemsep=0ex,partopsep=0ex]
	\item \label{B1} For all $i \in \{1,\dots,N\}$, $f(\cdot,i)$ is bounded from below by some $\bar f$ and the gradient $\nabla f(\cdot,i)$ is Lipschitz continuous on $\Rn$ with modulus ${\sf L} > 0$.
	\item \label{B2} The step sizes $\{\alpha_k\}_{k \geq 0}$ satisfy $\sum_{k=0}^\infty \alpha_k = \infty$ and $\sum_{k=0}^\infty \alpha_k^3 < \infty$.
\end{enumerate}

A detailed derivation of the steps shown in \Cref{sec:steps} for $\RR$ is deferred to \Cref{appen:RR}. Based on the discussion in \Cref{appen:RR} and on \cref{thm:convergence theorem}, we obtain the following results for $\RR$.
 
\begin{corollary}\label{thm:RR}
	We consider $\RR$ \eqref{eq:RR} for smooth nonconvex optimization problems under \ref{B1}--\ref{B2}. Then it holds that $\lim_{k \to \infty} \Exp[\|\nabla f(\vx^k)\|] = 0$ and $\lim_{k \to \infty} \|\nabla f(\vx^k)\| = 0$ almost surely. 
\end{corollary}

%\section{New results for stochastic optimization methods} \label{sec:new-results}
%In this section, we apply our unified convergence theorem to derive a set of novel convergence results for nonsmooth nonconvex optimization problems. %Our results offer new perspectives on the asymptotic behavior of several important stochastic nonsmooth methodologies. %optimization 
%obtained in this section are new to our knowledge. 

\subsection{Convergence of the proximal stochastic gradient method}\label{sec:prox-SGD}

We consider the composite-type optimization problem
\e \label{eq:comp-prob} {\min}_{\vx \in \Rn}~\psi(\vx) := f(\vx) + \varphi(\vx) \ee
where $f : \Rn \to \R$ is a continuously differentiable function and $\varphi : \Rn \to \Rex$ is $\tau$-weakly convex (see \Cref{appen:wcvx function}), proper, and lower semicontinuous. In this section, we want to apply our unified framework to study the convergence behavior of the well-known proximal stochastic gradient method ($\PSGD$):
%Examples for challenging large-scale and nonconvex applications that fit within our proposed framework comprise matrix decomposition \cite{CanRec09,ChaSanParWil09}, deep learning \cite{DenYu14,he2016deep,,MasBaxBarFre99,Sch15,simonyan2014very}, 
%
%Composite problems of the form \eqref{eq:comp-prob} are ubiquitous in machine learning \cite{Bis06,BotCurNoc18,shalev2014understanding,shi2010fast,lecun2015deep}, statistical learning and sparse regression \cite{friedman2001elements,SST2011}, dictionary learning \cite{BacJenMaiObo11}, image and signal processing \cite{ComPes11,ChaPoc11}, and in many other areas. 
%
%\cite{REF08a,SST2011,shi2010fast}
% 
%\textcolor{cuhkpl}{Literature and applications: why important?}
%
\begin{equation}  \vx^{k+1} = \prox_{\alpha_k\varphi}(\vx^k - \alpha_k \vg^k), \label{eq:update} \end{equation}
%
%\begin{align} \nonumber \vx^{k+1} & = (1-\alpha_k)\vx^k + \alpha_k \prox_{\lambda_k\varphi}(\vx^k - \lambda_k \vg^k)  \\ & = \vx^k - \alpha_k F_{\mathrm{nat}}^{\lambda_k}(\vx^k) + \alpha_k [\prox_{\lambda_k\varphi}(\vx^k - \lambda_k \vg^k)-\prox_{\lambda_k\varphi}(\vx^k - \lambda_k \nabla f(\vx^k))], \label{eq:update} \end{align}
%
where $\vg^k \approx \nabla f(\vx^k)$ is a stochastic approximation of $\nabla f(\vx^k)$, $\{\alpha_k\}_{k\geq 0} \subseteq \R_+$ is a suitable step size sequence, and $\prox_{\alpha_k \varphi} : \R^n \to \R^n$, $\prox_{\alpha_k \varphi}(\vx) := \argmin_{\vy \in \Rn} \varphi(\vy) + \frac{1}{2\alpha_k}\|\vx-\vy\|^2$ is the well-known proximity operator of $\varphi$. %In the special case $\alpha_k \equiv 1$, $k \in \N$, this algorithmic scheme reduces to the standard $\PSGD$ method.

\subsubsection{Assumptions and preparations}
We first recall several useful concepts from nonsmooth and variational analysis.  For a function $h : \Rn \to \Rex$, the Fr\'{e}chet (or regular) subdifferential of $h$ at the point $\vx$ is given by
\[ \partial h(\vx) := \{\vg \in \Rn : h(\vy) \geq h(\vx) + \iprod{\vg}{\vy-\vx}+o(\|\vy-\vx\|) \; \text{as} \; \vy \to \vx\}, \] 
see, e.g., \cite[Chapter 8]{RocWet98}. If $h$ is convex, then the Fr\'{e}chet subdifferential coincides with the standard (convex) subdifferential. 
%In addition, if $h$ is $\tau$-weakly convex, proper, and lower semicontinuous, then the proximity operator $\prox_{\alpha h}$ can be equivalently characterized via the optimality condition:
%%
%\begin{equation} \label{eq:let-it-be-prox} \vp = \prox_{\alpha h}(\vx) \quad \iff \quad \vx - \vp \in \alpha \partial h(\vp) \end{equation}
%%
%for all $\alpha \in (0,\tau^{-1})$. Furthermore, the proximal mapping $\prox_{\alpha h} : \R^n \to \R^n$ is Lipschitz continuous on $\R^n$ with Lipschitz constant $(1-\alpha\tau)^{-1}$, see, e.g., \cite[Proposition 12.19]{RocWet98} or \cite[Proposition 3.3]{HohLabObe20}.
%
%Based on the characterization \eqref{eq:let-it-be-prox}, 
%
It is well-known that the associated first-order optimality condition for the composite problem \eqref{eq:comp-prob} --- $0 \in \partial \psi(\vx) = \nabla f(\vx) + \partial\varphi(\vx)$ --- can be represented as a nonsmooth equation, \cite{RocWet98,HohLabObe20},
\[ F_{\mathrm{nat}}^\alpha(\vx) := \vx - \prox_{\alpha\varphi}(\vx - \alpha \nabla f(\vx)) = 0, \quad \alpha \in(0,\tau^{-1}), \]
where $F_{\mathrm{nat}}^\alpha$ denotes the so-called \textit{natural residual}. The natural residual $F_{\mathrm{nat}}^\alpha$ is a common stationarity measure for the nonsmooth problem \eqref{eq:comp-prob} and widely used in the analysis of proximal methods. %We will also use $F_{\mathrm{nat}}$ to denote the natural residual $F_{\mathrm{nat}}^1$ with the choice $\lambda = 1$. 

%Utilizing \cite[Exercise 8.8]{RocWet98}, the first-order optimality conditions of the composite problem \eqref{eq:comp-prob} can be expressed as
%
%\[ . \]
%
%... 
%
%If $h$ is continuously differentiable (on $\Rn$), then weak convexity of $h$ is equivalent to the following conditions
%
%\[ \left[ \begin{array}{l} h(\vy) - h(\vx) \geq \iprod{\nabla h(\vx)}{\vy-\vx} - \frac{\rho}{2}\|\vy-\vx\|^2 \quad \forall~\vx,\vy \in \Rn, \\[0.5ex] \iprod{\nabla h(\vx)-\nabla h(\vy)}{\vx-\vy} \geq - \rho \|\vx-\vy\|^2 \quad \forall~\vx,\vy \in \Rn, \end{array} \right. \]
%
%see, e.g., \cite{Via83,DavDru19}. 

%Furthermore, for a $\rho$-weakly convex, proper, and lower semicontinuous function $h : \Rn \to \Rex$, the proximity operator $\prox_{\lambda h} : \Rn \to \Rn$ is given by $\prox_{\lambda h}(\vx) := \argmin_{\vy \in \Rn}\,h(\vy) + \frac{1}{2\lambda}\|\vx-\vy\|^2$. 

We will make the following assumptions on $f$, $\varphi$, and the stochastic oracles $\{\vg^k\}_{k \geq 0}$:

%\begin{assumption}\label{ass-1} We consider the conditions:
	\begin{enumerate}[label=\textup{\textrm{(C.\arabic*)}},topsep=0pt,itemsep=0ex,partopsep=0ex]
		\item \label{C1} The function $f$ is bounded from below on $\Rn$, i.e., there is $\bar f$ such that $f(\vx) \geq \bar f$ for all $\vx \in \Rn$, and the gradient mapping $\nabla f$ is Lipschitz continuous (on $\Rn$) with modulus ${\sf L} > 0$. 
		\item \label{C2} The function $\varphi$ is $\tau$-weakly convex, proper, lower semicontinuous, and bounded from below on $\dom\varphi$, i.e., we have $\varphi(\vx) \geq \bar\varphi$ for all $\vx \in \dom\varphi$. 
		\item\label{C3}There exists ${\sf L}_\varphi > 0$ such that $\varphi(\vx)-\varphi(\vy) \leq {\sf L}_\varphi\|\vx-\vy\|$ for all $\vx,\vy \in \dom\varphi$.
		%\item \label{A2} The objective function $\psi$ is bounded from below on $\dom\varphi$, i.e., there exists a constant $\bar\psi \in \R$ such that $\psi(\vx) \geq \bar\psi$ for all $\vx \in \dom\varphi$.
%	\end{enumerate}
%\end{assumption}
%
%\begin{assumption}\label{ass-2} We will work with the following stochastic conditions:
%	\begin{enumerate}[label=\textup{\textrm{(B.\arabic*)}},topsep=0pt,itemsep=0ex,partopsep=0ex]
		\item \label{C4} Each $\vg^k$ defines an unbiased estimator of $\nabla f(\vx^k)$, i.e., we have $\Exp[\vg^k \mid \mathcal F_{k}] = \nabla f(\vx^k)$ almost surely, and there exist ${\sf C}, {\sf D} \geq 0$ such that 
		\[ \Exp[\|\vg^k-\nabla f(\vx^k)\|^2 \mid \mathcal F_{k}] \leq {\sf C}[f(\vx^k)-\bar f] + {\sf D} \quad \text{almost surely} \quad  \forall~k\in \N. \] 
		\item \label{C5} The step sizes $\{\alpha_k\}_{k \geq 0}$ satisfy $\sum_{k=0}^\infty \alpha_k = \infty$ and $\sum_{k=0}^\infty \alpha_k^2 < \infty$.
%	\end{enumerate}
%\end{assumption}
%
%\begin{assumption}\label{ass-3} We consider the additional assumption:
%	\begin{enumerate}[label=\textup{\textrm{(C.\arabic*)}},topsep=0pt,itemsep=0ex,partopsep=0ex]
		
		%\item \label{C2} The mapping $f$ is bounded from below on $\Rn$, i.e., there is $\bar f$ such that $f(\vx) \geq \bar f$ for all $\vx \in \Rn$.
	\end{enumerate}
%\end{assumption}

%Recall that a function $h : \Rn \to \Rex$ is called $\rho$-weakly convex if the mapping $h + \frac{\rho}{2}\|\cdot\|^2$ is convex (on $\dom h$). If $h$ is continuously differentiable (on $\Rn$), then weak convexity of $h$ is equivalent to the following conditions
%
%\[ \left[ \begin{array}{l} h(\vy) - h(\vx) \geq \iprod{\nabla h(\vx)}{\vy-\vx} - \frac{\rho}{2}\|\vy-\vx\|^2 \quad \forall~\vx,\vy \in \Rn, \\[0.5ex] \iprod{\nabla h(\vx)-\nabla h(\vy)}{\vx-\vy} \geq - \rho \|\vx-\vy\|^2 \quad \forall~\vx,\vy \in \Rn, \end{array} \right. \]
%
%see, e.g., \cite{Via83,DavDru19}. Furthermore, for a $\rho$-weakly convex, proper, and lower semicontinuous function $h : \Rn \to \Rex$, the proximity operator $\prox_{\lambda h} : \Rn \to \Rn$ is given by $\prox_{\lambda h}(\vx) := \argmin_{\vy \in \Rn}\,h(\vy) + \frac{1}{2\lambda}\|\vx-\vy\|^2$. 

%This expression is well-defined (and unique) for all $\lambda \in (0,\rho^{-1})$ and can be equivalently characterized via the optimality condition:
%
%\[ \vp = \prox_{\lambda h}(\vx) \quad \iff \quad \vx - \vp \in \lambda \partial h(\vp). \]  
%
%Using this characterization, the associated first-order optimality conditions for problem \eqref{eq:comp-prob} can be represented as a nonsmooth equation 
%
%\[ F_{\mathrm{nat}}^\lambda(\vx) := \vx - \prox_{\lambda\varphi}(\vx - \lambda \nabla f(\vx)) = 0, \quad \lambda > 0, \]
% 
%where $F_{\mathrm{nat}}^\lambda$ denotes the so-called \textit{natural residual}. We will also use $F_{\mathrm{nat}}$ to denote the natural residual $F_{\mathrm{nat}}^1$ with the choice $\lambda = 1$. 

Here, we again assume that the generated stochastic processes $\{\vx^k\}_{k \geq 0}$ is adapted to the filtration $\{\mathcal F_k\}_{k \geq 0}$. The assumptions \ref{C1}, \ref{C2}, \ref{C4}, and \ref{C5} are fairly standard and broadly applicable. In particular, \ref{C1}, \ref{C4}, and \ref{C5} coincide with the conditions \ref{A1}--\ref{A4} used in the analysis of $\SGD$. We continue with several remarks concerning condition \ref{C3}. 

\begin{remark} \label{rem:examples} Assumption \ref{C3} requires the mapping $\varphi$ to be Lipschitz continuous on its effective domain $\dom\varphi$. This condition holds in many important applications, e.g., when $\varphi$ is chosen as a norm or indicator function. Nonconvex examples satisfying \ref{C2} and \ref{C3} include, e.g., the minimax concave penalty (MCP) function \cite{Zha10}, the smoothly clipped absolute deviation (SCAD) \cite{Fan97}, or the student-t loss function. We refer to \cite{BoeWri21} and \cref{sec:some-examples} for further discussion. 
%
% \textcolor{cuhkpl}{Comparison literature. Nonconvex examples.}
\end{remark}

\subsubsection{Convergence results of prox-SGD}

We now analyze the convergence of the random process $\{\vx^k\}_{k\geq 0}$ generated by the stochastic algorithmic scheme \eqref{eq:update}. As pioneered in \cite{DavDru19}, we will use the Moreau envelope $\env_{\theta\psi}$, 
\e\label{eq:moreau env} \env_{\theta\psi}: \Rn \to \R, \quad \env_{\theta\psi}(\vx) := {\min}_{\vy\in\Rn}~\psi(\vy)+\frac{1}{2\theta}\|\vx-\vy\|^2, \ee
as a smooth Lyapunov function to study the descent properties and convergence of $\PSGD$.

We first note that the conditions \ref{C1} and \ref{C2} imply $\theta^{-1}$-weak convexity of $\psi$ for every $\theta \in (0,({\sf L}+\tau)^{-1}]$. In this case, the Moreau envelope $\env_{\theta\psi}$ is a well-defined and continuously differentiable function with gradient $\nabla\env_{\theta\psi}(\vx) = \frac{1}{\theta}(\vx-\prox_{\theta\psi}(\vx))$; see, e.g., \cite[Theorem 31.5]{Roc97}. 

As shown in \cite{DruLew18,DavDru19}, the norm of the Moreau envelope --- $\|\nabla \env_{\theta\psi}(x)\|$ --- defines an alternative stationarity measure for problem \eqref{eq:comp-prob} that is equivalent to the natural residual if $\theta$ is chosen sufficiently small. A more explicit derivation of this connection is provided in \cref{lem:env-to-nat}.

Next, we establish convergence of $\PSGD$ by setting $\vPhi \equiv \nabla \env_{\theta\psi}$ and $\mu_k \equiv \alpha_k$. Our analysis is based on the following two estimates which are verified in \cref{app:env} and \cref{sec:prox-sgd-expectation}.

\begin{lemma} \label{lem:psgd-1} Let $\{\vx^k\}_{k \geq 0}$ be generated by $\PSGD$ and let the assumptions \ref{C1}--\ref{C4} be satisfied. Then, for $\theta \in (0,[3{\sf L}+\tau]^{-1})$ and all $k$ with $\alpha_k \leq \min\{\frac{1}{2\tau},\frac{1}{2(\theta^{-1}-[{\sf L}+\tau])}\}$, it holds that 
\begin{align} \nonumber \Exp[\env_{\theta\psi}(\vx^{k+1}) - \bar\psi \mid \mathcal F_k] & \leq (1 + 4{\sf C}\theta^{-1}\alpha_k^2) \cdot [\env_{\theta\psi}(\vx^{k})-\bar\psi] \\ & \hspace{4ex} - {\sf L}\theta \alpha_k \|\nabla\env_{\theta\psi}(\vx^k)\|^2 + 2\alpha_k^2 ( {\sf C}{\sf L}_\varphi^2 + {\sf D}{\theta^{-1}}),  \label{eq:env-esti}\end{align}
almost surely, where $\bar \psi := \bar f + \bar \varphi$.
\end{lemma}

\begin{lemma} \label{lem:psgd-2} Let $\{\vx^k\}_{k \geq 0}$ be generated by $\PSGD$ and suppose that the assumptions \ref{C1}--\ref{C4} hold. Then, for $\theta \in (0,[\frac43{\sf L}+\tau]^{-1})$ and all $k$ with $\alpha_k \leq \frac{1}{2\tau}$, we have almost surely
\e \Exp[\|\vx^{k+1}-\vx^k\|^2 \mid \mathcal F_k]  \leq 8(2{\sf L}+{\sf C})\alpha_k^2 \cdot [\env_{\theta\psi}(\vx^k)-\bar\psi] +  4(((2{\sf L}+{\sf C})\theta+1){\sf L}_\varphi^2+{\sf D})\alpha_k^2. \label{eq:hello-hello} \ee
\end{lemma}

%In particular, we have
%
%\[ (1-3{\sf L}\theta){\gamma^{-1}} \|F_{\mathrm{nat}}(\vx)\| \leq \|\nabla\env_{\theta\psi}(\vx)\| \leq {(1+{\sf L}\theta)}{\gamma^{-1}} \|F_{\mathrm{nat}}(\vx)\|, \]  
%
%where $F_{\mathrm{nat}} := F^1_{\mathrm{nat}}$, $\gamma =$, and . (A ...)
%

\textbf{Phase I: Verifying  \ref{P1}--\ref{P2}.} In \cite[Corollary 3.4]{HohLabObe20}, it is shown that the gradient of the Moreau envelope is Lipschitz continuous with modulus ${\sf L}_e := \max\{{\theta^{-1}},(1-[{\sf L}+\tau]\theta)^{-1}{[{\sf L}+\tau]}\}$ for all $\theta \in (0,[{\sf L}+\tau]^{-1})$.
%Since the proximity operator $\prox_{\theta\psi}$ is Lipschitz continuous with constant $(1-[{\sf L}+\tau]\theta)^{-1}$, we can infer that $\nabla \env_{\theta\psi}$ is Lipschitz continuous with modulus , see alsofor comparison. 
Thus, condition \ref{P1} is satisfied. %Furthermore, utilizing the definition of $\env_{\theta\psi}$, we obtain the core estimate
%
%\begin{align} \nonumber \env_{\theta\psi}(\vx^{k+1}) & = \min_{\vy \in \Rn}~\psi(\vy) + \frac{1}{2\theta} \|\vy - \vx^{k+1}\|^2 \\ & \leq \env_{\theta\psi}(\vx^{k}) + \frac{1}{2\theta}\left[ \|\bar\vx^k - \vx^{k+1}\|^2 - \|\bar\vx^k - \vx^{k}\|^2 \right], \label{eq:a-core-estimate} \end{align}
%
%where $\bar\vx^k := \prox_{\theta\psi}(\vx^k)$. Setting $\bar \psi := \bar f + \bar \varphi$ and 
%

Furthermore, due to $\alpha_k \to 0$ and choosing $\theta \in (0,[3{\sf L}+\tau]^{-1})$, the estimate \eqref{eq:env-esti} in \cref{lem:psgd-1} holds for all $k$ sufficiently large. Consequently, due to $\env_{\theta\psi}(\vx) \geq \psi(\prox_{\theta\psi}(\vx)) \geq \bar \psi$ and \ref{C5},  \cref{Theorem:martingale_convergece} is applicable and upon taking total expectation, $\{\Exp[\env_{\theta\psi}(\vx^k)]\}_{k \geq 0}$ converges to some ${\sf E} \in \R$. In addition, the sequence $\{\env_{\theta\psi}(\vx^k)\}_{k \geq 0}$ converges almost surely to some random variable $e^\star$ and we have $\sum_{k=0}^\infty \alpha_k \Exp[\|\nabla \env_{\theta\psi}(\vx^k)\|^2] < \infty$. This verifies condition \ref{P2} with $a = 2$.

\textbf{Phase II: Verifying  \ref{P3}--\ref{P4} for showing convergence in expectation.} Assumptions \ref{C1}--\ref{C5} and \cref{lem:psgd-2} allow us to establish the required bound stated in \ref{P3}. Specifically, taking total expectation in \eqref{eq:hello-hello}, we have
\[ \Exp[\|\vx^{k+1}-\vx^k\|^2]  \leq 8(2{\sf L}+{\sf C})\alpha_k^2 \cdot \Exp[\env_{\theta\psi}(\vx^k)-\bar\psi] +  4(((2{\sf L}+{\sf C})\theta+1){\sf L}_\varphi^2+{\sf D})\alpha_k^2 \]
for all $k$ sufficiently large. %A detailed derivation of this estimate is provided in \cref{sec:prox-sgd-expectation}. 
Due to $\Exp[\env_{\theta\psi}(\vx^k)] \to {\sf E}$, there exists ${\sf F}$ such that $\Exp[\env_{\theta\psi}(\vx^k) - \bar\psi] \leq {\sf F}$ for all $k$. Hence, \ref{P3} holds with $q=2$, ${\sf A} = 8(2{\sf L}+{\sf C}){\sf F} + 4(((2{\sf L}+{\sf C})\theta + 1){\sf L}_\varphi^2 + {\sf D})$, $p_1 = 2$, and ${\sf B} = 0$. The property \ref{P4} easily follows from \ref{C5} and the parameter choices. Consequently, using \cref{thm:convergence theorem}, we can infer $\Exp[\|\nabla \env_{\theta\psi}(\vx^k)\|] \to 0$.

\textbf{Phase III: Verifying  \ref{P3'}--\ref{P4'} for showing almost sure convergence.} We follow the construction in \eqref{eq:decom update} and set $\bA_k = \alpha_k^{-1} (\vx^{k+1}-\vx^k-\Exp[\vx^{k+1}-\vx^k \mid \mathcal F_k])$, $\bB_k = \alpha_k^{-1} \Exp[\vx^{k+1}-\vx^k \mid \mathcal F_k]$, and $p_1, p_2 = 1$.
%
%First, let us introduce the following stochastic version of the natural residual $F_k(\vx^k) := \vx^k - \prox_{\alpha_k\varphi}(\vx^k-\alpha_k\vg^k)$. Then, the update \eqref{eq:update} can be written in the form:
%
%\[ \vx^{k+1} = \vx^k - \alpha_k \cdot \frac{1}{\alpha_k}\Exp[F_k(\vx^k) \mid \mathcal F_k] + \alpha_k \cdot \frac{1}{\alpha_k}[\Exp[F_k(\vx^k)\mid \mathcal F_k]-F_k(\vx^k)]. \] 
%
%We now set $p_1 = 1$, $\bA_k = \alpha_k^{-1} [\Exp[F_k(\vx^k)\mid \mathcal F_k]-F_k(\vx^k)]$, $p_2 = 1$, and $\bB_k = \alpha_k^{-1} \Exp[F_k(\vx^k) \mid \mathcal F_k]$. 
Clearly, we have $\Exp[\bA_k \mid \mathcal F_k] = 0$ and %using Jensen's inequality (for conditional expectations) and the law of total expectation, it holds that
%
%\[ \Exp[\|\bA_k\|^2] \leq 2\alpha_k^{-2} \Exp[\|\Exp[F_k(\vx^k)\mid \mathcal F_k]\|^2] + 2\alpha_k^{-2} \Exp[\|F_k(\vx^k)\|^2] \leq 4\alpha_k^{-2}\Exp[\|\vx^{k+1}-\vx^k\|^2]. \]  
%
based on the previous results in \textbf{Phase II}, we can show $\Exp[\|\vx^{k+1}-\vx^k\|^2] = \mathcal O(\alpha_k^2)$ which establishes boundedness of $\{\Exp[\|\bA_k\|^2\}_{k \geq 0}$. Similarly, for $\bB_k$ and by \cref{lem:psgd-2} and Jensen's inequality, we obtain
\[ \|\bB_k\|^2 \leq \alpha_k^{-2} \Exp[\|\vx^{k+1}-\vx^k\|^2 \mid \mathcal F_k] \leq 8(2{\sf L}+{\sf C}) \cdot [\env_{\theta\psi}(\vx^k)-\bar\psi] + \mathcal O(1). \]
Due to $\env_{\theta\psi}(\vx^k) \to e^\star$ almost surely, this shows $\limsup_{k \to \infty} \|\bB_k\|^2 < \infty$ almost surely. Hence, all requirements in \ref{P3'} are satisfied with $q=2$ and $b = 0$. Moreover, it is easy to see that property \ref{P4'} also holds in this case. Overall, \cref{thm:convergence theorem} implies $\|\nabla \env_{\theta\psi}(\vx^k)\| \to 0$ almost surely.

As mentioned, it is possible to express the obtained convergence results in terms of the natural residual $F_{\mathrm{nat}} = F_{\mathrm{nat}}^1$, see, e.g., \cref{lem:env-to-nat}. We summarize our observations in the following corollary.
 
\begin{corollary}\label{thm:prox-sgd}
	Let us consider $\PSGD$ \eqref{eq:update} for the composite problem \eqref{eq:comp-prob} under \ref{C1}--\ref{C5}. Then, we have $\lim_{k \to \infty} \Exp[\|F_{\mathrm{nat}}(\vx^k)\|] = 0$ and $\lim_{k \to \infty} \|F_{\mathrm{nat}}(\vx^k)\| = 0$ almost surely. 
\end{corollary}

\begin{remark}
	As a byproduct, \Cref{lem:psgd-1} also leads to an expected iteration complexity result of $\PSGD$ by using the ABC condition \ref{C4} rather than the standard bounded variance assumption. This is a nontrivial extension of \cite[Corollary 3.6]{DavDru19}. We provide a full derivation in \Cref{appen:complexity PSGD}. 
\end{remark}

\subsection{Convergence of stochastic model-based methods}\label{sec:SMM}

In this section, we consider the convergence of stochastic model-based methods for nonsmooth weakly convex optimization problems
\e \label{eq:wcvx prob} 
{\min}_{\vx \in \Rn}~\psi(\vx) := f(\vx) + \varphi(\vx) = \Exp_{\xi\sim D}[f(\vx,\xi)] + \varphi(\vx), 
\ee
where both $f$ and $\varphi$ are assumed to be (nonsmooth) weakly convex functions and $\psi$ is lower bounded, i.e., $\psi(\vx)\geq \bar \psi$ for all $\vx\in \dom \varphi$. Classical stochastic optimization methods --- including proximal stochastic subgradient, stochastic proximal point, and stochastic prox-linear methods --- for solving \eqref{eq:wcvx prob} are unified by the stochastic model-based methods ($\SMM$) \cite{DucRua18,DavDru19}:
\e\label{eq:model based}
\vx^{k+1} = {\argmin}_{\vx \in \Rn} \ f_{\vx^k}(\vx, \xi^{k})  +   \varphi(\vx) + \frac{1}{2\alpha_k}  \|\vx - \vx^k\|^2,
\ee
where $ f_{\vx^k}(\vx, \xi^{k})$ is a stochastic approximation of  $f$ around $\vx^k$ using the sample $\xi^{k}$; see \Cref{appen:three SMM} for descriptions of three major types of $\SMM$. Setting $\calF_k := \sigma(\xi^0,\ldots, \xi^{k-1})$, it is easy to see that $\{\vx^{k}\}_{k\geq 0}$ is adapted to $\{\calF_k\}_{k\geq 0}$.
%
%\subsubsection{Assumptions}
%
We analyze convergence of $\SMM$ under the following assumptions.
\begin{enumerate}[label=\textup{\textrm{(D.\arabic*)}},topsep=0pt,itemsep=0ex,partopsep=0ex]
		\item \label{D1} The stochastic approximation function $f_{\vx}$ satisfies a one-sided accuracy property, i.e., we have $\Exp_\xi[f_{\vx}(\vx,\xi) ] = f(\vx)$ for all $\vx \in U$ and 
		\[
		    \Exp_{\xi}[f_{\vx}(\vy,\xi) - f(\vy)] \leq \frac{\tau}{2} \|\vx-\vy\|^2 \quad \forall~\vx,\vy  \in U,
		\]
		 where $U$ is an open convex set containing $\dom \varphi$.
		\item \label{D2} The function $\vy \mapsto f_{\vx}(\vy,\xi)+\varphi(\vy)$ is $\eta$-weakly convex for all $\vx\in U$ and almost every $\xi$.
		\item \label{D3} There exists ${\sf L}>0$ such that the stochastic approximation function $f_{\vx}$ satisfies 
		\[
		     f_{\vx}(\vx,\xi) - f_{\vx}(\vy,\xi) \leq {\sf L} \|\vx-\vy\| \quad \forall~\vx,\vy\in U, \quad \text{and almost every $\xi$}.
		\]
		\item \label{D4} The function $\varphi$ is ${\sf L}_\varphi$-Lipschitz continuous. 
		\item \label{D5} The step sizes $\{\alpha_k\}_{k \geq 0}$ satisfy $\sum_{k=0}^\infty \alpha_k = \infty$ and $\sum_{k=0}^\infty \alpha_k^2 < \infty$. 
\end{enumerate}

Assumptions \ref{D1}, \ref{D2},  \ref{D3} are standard for analyzing $\SMM$ and identical to that of \cite{DavDru19}. \ref{D5} is convention for stochastic methods. Assumption \ref{D4} mimics \ref{C3}; see \Cref{rem:examples} for discussions.

%\subsubsection{Convergence results of SMM}

We now derive the convergence of $\SMM$ below by setting $\vPhi \equiv \nabla \env_{\theta\psi} $ and $\mu_k \equiv \alpha_k$. Our derivation is based on the following two estimates, in which the proof of \Cref{lemma:step length model-based} is given in \Cref{appen:SMM step length}.
\begin{lemma}[Theorem 4.3 of \cite{DavDru19}]\label{lemma:recursion model-based}
	Let $\theta \in (0,(\tau+\eta)^{-1})$ and $\alpha_k <  \theta$ be given. Then, we have 
	\[
	\Exp[\env_{\theta\psi}(\vx^{k+1}) \mid \calF_k] \leq \env_{\theta\psi}(\vx^{k})  - \frac{(1 - [\tau + \eta]\theta)\alpha_{k}}{2(1-\eta\alpha_{k})}  \|\nabla \env_{\theta\psi} (\vx^k)\|^2 + \frac{2{\sf L}^2\alpha_{k}^2}{(1 - \eta\alpha_{k}) (\theta -\alpha_{k})}. 
	\]
\end{lemma}

\begin{lemma}\label{lemma:step length model-based}
    For all $k$ with $\alpha_k \leq 1/{(2\eta)}$, it holds that
	\[
	\Exp[\|\vx^{k+1} - \vx^k\|^2 \mid \calF_k]  \leq (16({\sf L} + {\sf L}_\varphi )^2 + 8{\sf L}^2) \alpha_k^2. 
	\] 
\end{lemma}

\textbf{Phase I: Verifying  \ref{P1}--\ref{P2}.} As before, \cite[Corollary 3.4]{HohLabObe20} implies that the mapping $\nabla \env_{\theta\psi}$ is Lipschitz continuous for all $\theta \in (0,(\tau+\eta)^{-1})$
%We can infer from \cite[Corollary 3.4]{HohLabObe20} that $\nabla \env_{\theta\psi}$ is Lipschitz continuous with modulus ${\sf L}_e := (1+(1-(\tau+\eta)\theta)^{-1})\theta^{-1}$ for all $\theta \in (0,(\tau+\eta)^{-1})$. 
Hence, condition \ref{P1} is satisfied. Using $\alpha_k \to 0$, we can apply \cref{Theorem:martingale_convergece} to the recursion obtained in \cref{lemma:recursion model-based} for all $k$ sufficiently large and it follows $\sum_{k=0}^\infty \alpha_k \Exp[\|\nabla \env_{\theta\psi}(\vx^k)\|^2] < \infty$.  Thus, condition \ref{P2} holds with $a = 2$.

\textbf{Phase II: Verifying  \ref{P3}--\ref{P4} for showing convergence in expectation.} Taking total expectation in \cref{lemma:step length model-based} verifies condition \ref{P3} with $q = 2$, ${\sf A} = (16({\sf L} + {\sf L}_\varphi )^2 + 8{\sf L}^2) $, $p_1 = 2$, ${\sf B} = 0$. 
Moreover, condition \ref{P4} is true by assumption \ref{D5} and the previous parameters choices. Thus, applying \Cref{thm:convergence theorem} gives $ \Exp[\|\nabla \env_{\theta\psi}(\vx^k)\|] \to 0$. 

\textbf{Phase III: Verifying  \ref{P3'}--\ref{P4'} for showing almost sure convergence.} 
As in \eqref{eq:decom update}, we can set $\bA_k = \alpha_k^{-1} (\vx^{k+1}-\vx^k-\Exp[\vx^{k+1}-\vx^k \mid \mathcal F_k])$, $\bB_k = \alpha_k^{-1} \Exp[\vx^{k+1}-\vx^k \mid \mathcal F_k]$. Applying \Cref{lemma:step length model-based} and utilizing Jensen's inequality, we have $\Exp[\mA_k \mid \calF_k] = 0$, $\Exp[\|\mA_k\|^2] \leq (4/\alpha_k^2) \Exp[\|\vx^{k+1} - \vx^k\|^2] \leq 4(16({\sf L} + {\sf L}_\varphi )^2 + 8{\sf L}^2)$ and $ \|\mB_k\|^2 \leq 16({\sf L} + {\sf L}_\varphi )^2 + 8{\sf L}^2$. Thus, condition \ref{P3'} is satisfied with $p_1 =p_2 = 1$, $q = 2$. Assumption \ref{D5}, together with  the previous parameter choices verifies condition \ref{P4'} and hence, applying \Cref{thm:convergence theorem} yields $\|\nabla \env_{\theta\psi}(\vx^k)\| \to 0$ almost surely. 

Summarizing this discussion, we obtain the following convergence results for $\SMM$. 
\begin{corollary}\label{thm:model based}
	We consider the family of stochastic model-based methods \eqref{eq:model based} for the optimization problem \eqref{eq:wcvx prob} under assumptions \ref{D1}--\ref{D5}. Let $\{\vx^k\}_{k\geq 0}$ be a generated sequence. Then, we have $\lim_{k \to \infty} \Exp[\|\nabla \env_{\theta\psi}(\vx^k)\| ] = 0$ and $\lim_{k \to \infty} \|\nabla \env_{\theta\psi}(\vx^k)\| = 0$ almost surely. 
\end{corollary}
%
%Consequently, every accumulation point of $\{\vx^k\}_{k\geq 0}$ is almost surely a stationary point of \eqref{eq:wcvx prob}.

\begin{remark}The results presented in \cref{thm:model based} also hold under certain extended settings. In fact, we can replace \ref{D3} by a slightly more general Lipschitz continuity assumption on $f$.  Moreover, it is possible to establish convergence in the case where $f$ is not Lipschitz continuous but has Lipschitz continuous gradient, which is particularly useful when we apply stochastic proximal point method for smooth $f$. A more detailed derivation and discussion of such extensions is deferred to \cref{app:smm-extension}.\vspace{-1ex}
\end{remark}

\subsection{Related work and discussion} \label{sec:literature}
\textbf{SGD and RR.} The literature for $\SGD$ is extremely rich and several connected and recent works have been discussed in \Cref{sec:intro}. Our result in \Cref{thm:SGD} unifies many of the existing convergence analyses of $\SGD$ and is based on the general ABC condition \ref{A3} (see \cite{khaled2020better,lei2019stochastic,GowSebLoi21} for comparison) rather than on the standard bounded variance assumption. Our expected convergence result generalizes the one in \cite{BotCurNoc18} using much weaker assumptions. Our results for $\RR$ are in line with the recent theoretical observations in \cite{mishchenko2020,nguyen2020unified,LiMiQiu21}. In particular, \Cref{thm:RR} recovers the almost sure convergence result shown in \cite{LiMiQiu21}, while the expected convergence result appears to be new. 

\textbf{Prox-SGD and SMM.} The work \cite{DavDru19} established one of the first complexity results for $\PSGD$ using the Moreau envelope. Under a bounded variance assumption (${\sf C} = 0$ in condition \ref{C4}) and for general nonconvex and smooth $f$, the authors showed
$\Exp[\|\nabla \env_{\theta\psi}(\vx^{\bar k})\|^2] = \mathcal O((T+1)^{-1/2})$,
where $\vx^{\bar k}$ is sampled uniformly from the past $T+1$ iterates $\vx^0,\dots,\vx^T$. As mentioned, this result cannot be easily extended to the asymptotic convergence results discussed in this paper. Earlier studies of $\PSGD$ for nonconvex $f$ and ${\sf C} = 0$ include \cite{GhaLanZha16} where convergence of $\PSGD$ is established if the variance parameter ${\sf D} = {\sf D}_k \to 0$ vanishes as $k \to \infty$. This can be achieved by progressively increasing the size of the selected mini-batches or via variance reduction techniques as in $\mathsf{prox}\text{-}\mathsf{SVRG}$ and $\mathsf{prox}\text{-}\mathsf{SAGA}$, see \cite{RedSraPocSmo16}.
The question whether $\PSGD$ can converge and whether the accumulation points of the iterates $\{\vx^k\}_{k \geq 0}$ correspond to stationary points was only addressed recently in \cite{MajMiaMou18}. The authors use a differential inclusion approach to study convergence of $\PSGD$. However, additional compact constraints $\vx \in \mathcal X$ have to be introduced in the model \eqref{eq:comp-prob} to guarantee sure boundedness of $\{\vx^k\}_{k\geq 0}$ and applicability of the differential inclusion techniques. Lipschitz continuity of $\varphi$ also appears as an essential requirement in \cite[Theorem 5.4]{MajMiaMou18}. The analyses in \cite{DucRua18,DavDruKakLee20} establish asymptotic convergence guarantees for $\SMM$. However, both works require a priori (almost) sure boundedness of $\{\vx^k\}_{k\geq 0}$ and a density / Sard-type condition in order to show convergence. We refer to \cite{GeiSca21} for an extension of the results in \cite{MajMiaMou18,DavDruKakLee20} to $\PSGD$ in Hilbert spaces. By contrast, our convergence framework allows to complement these differential inclusion-based results and --- for the first time --- fully removes any stringent boundedness assumption on $\{\vx^k\}_{k \geq 0}$. Instead, our analysis relies on more transparent assumptions that are verifiable and common in stochastic optimization and machine learning. In summary, we are now able to claim: \textit{$\PSGD$ and $\SMM$ converge under standard stochastic conditions if $\varphi$ is Lipschitz continuous}. In the easier convex case, analogous results have been obtained, e.g., in \cite{GhaLanZha16,AtcForMou17,RosVilVu20}. 

 We provide an overview of several related and representative results in \Cref{table:literature} in \Cref{appen:table}. \vspace{-1ex}

\section{Conclusion}\label{sec:conclusion}
In this work, we provided a novel convergence framework that allows to derive expected and almost sure convergence results for a vast class of stochastic optimization methods under state-of-the-art assumptions and in a unified way. We specified the steps on how to utilize our theorem in order to establish convergence results for a given stochastic algorithm. As concrete examples, we applied our theorem to derive asymptotic convergence guarantees for $\SGD$, $\RR$, $\PSGD$, and $\SMM$. To our surprise, some of the obtained results appear to be new and provide new insights into the convergence behavior of some well-known and standard stochastic methodologies. 
%under standard assumptions used in stochastic optimization and machine learning societies. 
%
These applications revealed that our unified theorem can serve as a plugin-type tool with the potential to facilitate the convergence analysis of a wide class of stochastic optimization methods. % such as stochastic momentum method, Adam, stochastic higher order methods, etc. 

%Our current framework is limited to algorithmic schemes that fit the conditions stated in \Cref{thm:convergence theorem}. Furthermore, i
Finally, it is important to investigate in which situations our convergence results in terms of the stationarity measure $\vPhi$ can be strengthened --- say to almost sure convergence guarantees for the iterates $\{\vx^k\}_{k \geq 0}$. We plan to consider such a possible extension in future work.

\begin{ack}
	The authors would like to thank the Area Chair and anonymous reviewers for their detailed and constructive comments, which have helped greatly to improve the quality and presentation of the manuscript. {In addition, we would like to thank Michael Ulbrich for valuable feedback and comments on an earlier version of this work.} 
	
	X. Li was partially supported by the National Natural Science Foundation of China (NSFC) under grant No. 12201534 and 72150002 and by the Shenzhen Science and Technology Program under Grant No. RCBS20210609103708017. A. Milzarek was partly supported by the National Natural Science Foundation of China (NSFC) -- Foreign Young Scholar Research Fund Project (RFIS) under Grant No. 12150410304 and by the Fundamental Research Fund -- Shenzhen Research Institute of Big Data (SRIBD) Startup Fund JCYJ-AM20190601.
\end{ack}

\bibliographystyle{plain}
\bibliography{references}

\newpage
{\small
\tableofcontents
}

\appendix

\newpage

\section{Proof of Theorem 2.1}\label{appen:proof main thm}
%\section{Proof of \Cref{thm:convergence theorem}}\label{appen:proof main thm}
	\paragraph{\underline{Part I: Proof of item (i).}} %Next, we prove convergence of $\{\Exp[\|\vPhi(\vx^k)\|]\}_{k \geq 0}$. 
First, combining $\sum_{k=0}^\infty \mu_k = \infty$ in \ref{P4} and \ref{P2}, we can infer $\liminf_{k\to\infty} \Exp[\|\vPhi(\vx^k)\|^a] = 0$. Let us assume that $\{\Exp[\|\vPhi(\vx^k)\|]\}_{k\geq 0}$ does not converge to zero. Then, there exist $\delta > 0$ and two subsequences $\{\ell_t\}_{t \geq 0}$ and $\{u_t\}_{t \geq 0}$ such that 
$\ell_t < u_t \leq \ell_{t+1}$ and
\e\label{eq:sequence construct} \Exp[\|\vPhi(\vx^{\ell_t})\|] \geq 2\delta, \quad \Exp[\|\vPhi(\vx^{u_t})\|^a] \leq \delta^a, \quad \text{and} \quad \Exp[\|\vPhi(\vx^k)\|^a] > \delta^a \ee
for all $\ell_t < k < u_t$ and $t \in \mathbb{N}$.  Combining \ref{P2} and \eqref{eq:sequence construct} and applying Jensen's inequality ($a \geq 1$), this yields
\[
\infty >  \sum_{t = 0}^{\infty } \ \sum_{k= \ell_t}^{u_t-1} \mu_k \, \Exp[\|\vPhi(\vx^k) \|^a] \geq \sum_{t = 0}^{\infty } \ \left[\sum_{k= \ell_t+1}^{u_t-1} \mu_k\delta^a + \mu_{\ell_t} (\Exp[\|\vPhi(\vx^{\ell_t})\|])^a \right] \geq \delta^a \sum_{t = 0}^{\infty } \ \sum_{k= \ell_t}^{u_t-1} \mu_k,
\]
which immediately implies $\beta_t := \sum_{k= \ell_t}^{u_t-1} \mu_k \to 0$. Since $\{\mu_k\}_{k \geq 0}$ is bounded, there exists $\bar\mu$ such that $\mu_k \leq \bar\mu$ for all $k$. For any $p \geq 1$, this further implies $\mu_k^p \leq \bar\mu^{p-1}\mu_k$ for all $k$. Using H\"older's and Jensen's inequality, $q \geq 1$, and \ref{P3}, we have %for all $t\geq 0$ sufficiently large
	\begin{align} \nonumber 
		\Exp[\|\vx^{u_t} - \vx^{\ell_t}\|] &\leq  \sum_{k=\ell_t}^{u_t - 1} \Exp[\|\vx^{k+1} - \vx^{k}\|] \leq \left[\sum_{k=\ell_t}^{u_t - 1} \mu_k \right]^{1-\frac{1}{q}} \left[\sum_{k=\ell_t}^{u_t - 1} \mu_k^{-(q-1)} \Exp[\|\vx^{k+1} - \vx^{k}\|^q] \right]^{\frac{1}{q}} 	\label{eq:step length}  \\
		&\leq \beta_t^{1-\frac{1}{q}} \left[{\sum}_{k=\ell_t}^{u_t - 1}  {\sf A}\mu_k^{p_1-(q-1)} + {\sf B}\mu_k^{p_2-(q-1)} \Exp[\|\vPhi(\vx^k)\|^b] \right]^\frac{1}{q} \\
		& \nonumber \leq \beta_t^{1-\frac{1}{q}} \left[{\sum}_{k=\ell_t}^{u_t - 1}  {\sf A}\bar\mu^{p_1-q}\mu_k + {\sf B}\bar\mu^{p_2-q}\mu_k  \Exp[\|\vPhi(\vx^k)\|^b] \right]^\frac{1}{q} \\
		&\nonumber = \beta_t^{1-\frac{1}{q}} \left[ {\sf A}\bar\mu^{p_1-q}\beta_t + {\sf B}\bar\mu^{p_2-q}{\sum}_{k=\ell_t}^{u_t - 1} \mu_k  \Exp[\|\vPhi(\vx^k)\|^b] \right]^\frac{1}{q},
	\end{align}
	where we have utilized $p_1,p_2 \geq q$ in the third line. We first consider the case $a > b$. We can apply Jensen's and H\"older's inequality to obtain
	\begingroup
	\allowdisplaybreaks
	\begin{align} \nonumber {\sum}_{k=\ell_t}^{u_t - 1} \mu_k \Exp[\|\vPhi(\vx^k)\|^b] & \leq {\sum}_{k=\ell_t}^{u_t - 1} \mu_k \Exp[\|\vPhi(\vx^k)\|^a]^\frac{b}{a} \\ & \leq \left[{\sum}_{k=\ell_t}^{u_t - 1} \mu_k\right]^\frac{a-b}{a} \left[{\sum}_{k=\ell_t}^{u_t - 1} \mu_k \, \Exp[\|\vPhi(\vx^k)\|^a] \right]^\frac{b}{a} \label{eq:a-nice-inequality} \\
		&\nonumber =\beta_t^\frac{a-b}{a} \left[{\sum}_{k=\ell_t}^{u_t - 1} \mu_k \, \Exp[\|\vPhi(\vx^k)\|^a] \right]^\frac{b}{a}. 
	\end{align}
	\endgroup
	Plugging \eqref{eq:a-nice-inequality} into \eqref{eq:step length} and invoking \ref{P2} and $ \beta_t \to 0$, we have $\Exp[\|\vx^{u_t}-\vx^{\ell_t}\|] \to 0$ as $t \to \infty$. We now consider the case $a = b$.  Let us introduce $E_k := \sum_{i=0}^{k-1} \mu_i \, \Exp[\|\vPhi(\vx^i)\|^a]$. By \ref{P2}, the sequence $\{E_k\}_{k\geq 1}$ converges and hence, we have $E_{u_t}-E_{\ell_t}\to 0$ as $t\to \infty$.  It then follows from \eqref{eq:step length} that 
	\[ \Exp[\|\vx^{u_t} - \vx^{\ell_t}\|] \leq \beta_t^{1-\frac{1}{q}} \left[ {\sf A}\beta_t + {\sf B}[E_{u_t}-E_{\ell_t}] \right]^\frac{1}{q} \to 0 \quad \text{as} \quad t\to \infty.\] 
	Together, this establishes $\Exp[\|\vx^{u_t} - \vx^{\ell_t}\|]  \to 0$ as $t \to \infty$ in both cases. However, by the Lipschitz continuity of $\vPhi$ in \ref{P1}, $a\geq 1$ in \ref{P4},  and the construction \eqref{eq:sequence construct},  we have 
\e\label{eq:contradiction 2}
\begin{split}
	\delta &\leq \Exp[\|\vPhi(\vx^{\ell_t})\|]  - \Exp[\|\vPhi(\vx^{u_t}) \|] \leq \Exp[\|\vPhi(\vx^{u_t}) - \vPhi(\vx^{\ell_t}) \|]  \leq {\sf L}_{\Phi} \Exp[\|\vx^{u_t} - \vx^{\ell}\|].
\end{split}
\ee
By letting $t\rightarrow \infty$ in \eqref{eq:contradiction 2}, we get a contradiction. This concludes the proof of item (i). 

\paragraph{\underline{Part II: Proof of item (ii).}} In order to control the stochastic behavior of the error terms $\bA_k$ and to establish the almost sure convergence of the sequence $\{\|\vPhi(\vx^k)\|\}_{k \geq 0}$, we will utilize several results from martingale theory. 

\begin{defn}[Martingale]
	\label{Def:mart}
	Let $(\Omega, \mathcal{U}, \Prob)$ be a probability space and $\{\mathcal{U}_{k}\}_{k \geq 0}$ a family of increasing sub-$\sigma$-fields of $\mathcal{U} .$ A random process $\left\{\bM_{k}\right\}_{k \geq 0}$ defined on this probability space is said to be a martingale with respect to the family $\{\mathcal{U}_{k}\}_{k \geq 0}$ (or an $\{\mathcal{U}_{k}\}_{k \geq 0}$-martingale) if each $\bM_k$ is integrable and $\mathcal{U}_{k}$-measurable and we have $\Exp[\bM_{k+1} \mid \mathcal{U}_{k}]=\bM_{k}$ a.s. for all $k$.
\end{defn}

Next, we state a standard convergence theorem for vector-valued martingales, see, e.g., \cite[Theorem 5.2.22 and Section 5.3]{Str11} or \cite[Theorem 5.14]{Bre92}.

\begin{thm}[Martingale Convergence Theorem]
	\label{Thm:mart_conv}
	Let $\{\bM_k\}_{k \geq 0}$ be a given vector-valued $\{\mathcal{U}_k\}_{k \geq 0}$-martingale as specified in \cref{Def:mart}. If $\sup_k \Exp[\|\bM_k\|] < \infty$, then $\{\bM_{k}\}_{k \geq 0}$ converges almost surely to an integrable random vector $\bM$.
\end{thm}

%We will also require the well-known Burkholder-Davis-Gundy inequality \cite{BurDavGun72,Str11} in our analysis. 
%
%\begin{thm}[Burkholder-Davis-Gundy Inequality]
%\label{Thm:BDG}
%Let $\{\bM_{k}, \mathcal{U}_{k}\}$, $k \geq 0$, be a given vector-valued martingale with $\bM_{0}=0$. Then, for all $p \in (1,\infty)$, there exists $C_p > 0$ such that
%
%\[ \Exp\left[{\sup}_{k \geq 0} \|\bM_k\|^p\right] \leq C_p \cdot \Exp\left[ \left({\sum}_{k=1}^\infty \|\bM_{k}-\bM_{k-1}\|^2 \right)^\frac{p}{2} \right]. \]
%\end{thm}
%
%Notice that the constant $C_p$ in \cref{Thm:BDG} is universal and does neither depend on the martingale $\{\bM_k\}$ nor on the dimension of the underlying space. 

\textbf{Step (a): Analysis of the error terms $\{\bA_k\}_{k \geq 0}$.} Let us define $\bM_k := \sum_{i=0}^{k-1}\mu_i^{p_1}\bA_i$. Then, it holds that
\[ \Exp[\bM_{k+1} \mid \mathcal F_k] = {\sum}_{i=0}^{k-1} \mu_{i}^{p_1}\bA_{i} + \mu_k^{p_1} \Exp[\bA_k \mid \mathcal F_k] = \bM_k,  \]
i.e., $\{\bM_k\}_{k \geq 0}$ defines a martingale adapted to the filtration $\{\mathcal F_k\}_{k \geq 0}$. In addition, inductively, we obtain 
\begin{align*} 
	\Exp[\|\bM_k\|^2] & = \Exp[\|\bM_{k-1}\|^2] + 2\mu_{k-1}^{p_1}\Exp[\iprod{\bM_{k-1}}{\bA_{k-1}}] +\mu_{k-1}^{2p_1} \Exp[\|\bA_{k-1}\|^2] \\ & = \cdots  = {\sum}_{i=0}^{k-1} \mu_i^{2p_1} \Exp[\|\bA_i\|^2] \leq {\sum}_{i=0}^{k-1} \mu_i^{2p_1} \Exp[\|\bA_i\|^q]^{\frac{2}{q}} \leq {\sf A}^{\frac{2}{q}} {\sum}_{i=0}^{k-1} \mu_i^{2p_1},
\end{align*}
where we utilized Jensen's inequality,  the concavity of the mapping $x \mapsto x^{\frac{2}{q}}$, and $\Exp[\|\bA_i\|^q]\leq {\sf A}$ in \ref{P3'}. Hence, by \ref{P4'} and Jensen's inequality, we can infer $\sup_{k} \Exp[\|\bM_k\|] \leq {\sf A}^\frac{1}{q}[{\sum}_{i=0}^{\infty} \mu_i^{2p_1}]^\frac12 < \infty$. \cref{Thm:mart_conv} then implies that $\{\bM_k\}_{k \geq 0}$ converges almost surely to some integrable random vector $\bM$.

Next, we establish almost sure convergence of $\{\|\vPhi(\vx^k)\|\}_{k \geq 0}$. Our derivation generally mimics the proof of item (i), but uses sample-based arguments.

\textbf{Step (b): Almost sure convergence.} First, applying the monotone convergence theorem to \ref{P2} gives $\Exp[\sum_{k=0}^\infty \mu_k\|\vPhi(\vx^k)\|^a] = \lim_{T \to \infty} \sum_{k=0}^T \mu_k \Exp[\|\vPhi(\vx^k)\|^a] < \infty$, which immediately implies $\sum_{k=0}^\infty \mu_k\|\vPhi(\vx^k)\|^a < \infty$ almost surely. This, together with $\sum_{k=0}^\infty \mu_k = \infty$, yields $\liminf_{k\to \infty} \|\vPhi(\vx^k)\| = 0$ almost surely. %In addition, using assumption ?, we obtain
%
%\[ \Exp\left[{\sum}_{k=0}^\infty \mu_k^{?}\|\bA_k\|^q\right] = \lim_{T \to \infty} {\sum}_{k=0}^T \mu_k^{?} \Exp[\|\bA_k\|^q] \leq {\sf A} \cdot {\sum}_{k=0}^\infty \mu_k^{?} < 0. \]
%
We now consider an arbitrary sample $\omega \in \mathcal M$ where
\begingroup
\allowdisplaybreaks
\begin{align*} \mathcal M &:= \left\{\omega: {\sum}_{k=0}^\infty \mu_k\|\vPhi(\vx^k(\omega))\|^a < \infty, \;\; \lim_{k\to\infty} \bM_k(\omega) = \bM(\omega), \right. \\ & \hspace{20ex} \left. \;\; \text{and} \;\; \limsup_{k \to \infty} \frac{\|\mB_k(\omega)\|^q}{1+\|\vPhi(\vx^k(\omega))\|^b} < \infty \right\}. \end{align*}
\endgroup
Our preceding discussion implies that the event $\mathcal M$ occurs with probability $1$ and it holds that $\liminf_{k\to \infty} \|\vPhi(\vx^k(\omega))\| = 0$. Let us assume that $\{\|\vPhi(\vx^k(\omega))\|\}_{k\geq 0}$ does not converge to zero. Then, there exist $\delta > 0$ and two subsequences $\{\ell_t\}_{t \geq 0}$ and $\{u_t\}_{t \geq 0}$ such that $\ell_t < u_t \leq \ell_{t+1}$ and
\e\label{eq:sequence-as} \|\vPhi(\vx^{\ell_t}(\omega))\| \geq 2\delta, \quad \|\vPhi(\vx^{u_t}(\omega))\| \leq \delta, \quad \text{and} \quad \|\vPhi(\vx^k(\omega))\| > \delta \ee
for all $\ell_t < k < u_t$ and $t \in \mathbb{N}$. (Notice that in contrast to the proof of item (i), the sequences $\{\ell_t\}_{t \geq 0}$, $\{u_t\}_{t \geq 0}$, and $\delta$ will now generally depend on the selected sample $\omega$). Due to $\omega \in \mathcal M$, this yields
\[ \infty >  {\sum}_{t = 0}^{\infty } \ {\sum}_{k= \ell_t}^{u_t-1} \mu_k \|\vPhi(\vx^k(\omega)) \|^a \geq \delta^a {\sum}_{t = 0}^{\infty } \ {\sum}_{k= \ell_t}^{u_t-1} \mu_k, \]
which again implies $\beta_t := \sum_{k=\ell_t}^{u_t-1}\mu_k \to 0$. Furthermore, there exist $K \in \N$ and $B \in \R$
such that $\|\mB_k(\omega)\| \leq B(1+\|\vPhi(\vx^k(\omega))\|^b)^{{1}/{q}}$ for all $k \geq K$. We first consider the case $qa > b$. It follows from \ref{P3'} that 
%
%\begingroup
%\allowdisplaybreaks
%\begin{align*} \|\vx^{u_t}(\omega)-\vx^{\ell_t}(\omega)\| & \leq B{\sum}_{k = \ell_t}^{u_t-1} \mu_k^{p_2}(1+\|\vPhi(\vx^k(\omega))\|^b)^\frac{1}{q} + \left\|{\sum}_{k = \ell_t}^{u_t-1} \mu_k^{p_1} \bA_k(\omega) \right\| \\ 
%&\leq B{\sum}_{k = \ell_t}^{u_t-1} \mu_k^{p_2}   + B{\sum}_{k = \ell_t}^{u_t-1} \mu_k^{p_2} \|\vPhi(\vx^k(\omega))\|^{\frac{b}{q}} + \left\|{\sum}_{k = \ell_t}^{u_t-1} \mu_k^{p_1} \bA_k(\omega) \right\| \\ 
%& \leq B\beta_t + B\left(\sum_{k = \ell_t}^{u_t-1} \mu_k^{\frac{qa}{qa-b}(p_2-1)+1}\right)^{1-\frac{b}{qa}}\left(\sum_{k = \ell_t}^{u_t-1} \mu_k \|\vPhi(\vx^k(\omega))\|^a\right)^\frac{b}{qa} \\ & \hspace{4ex} + \|\bM_{u_t}(\omega)-\bM_{\ell_t}(\omega)\| \\ & \leq B\beta_t + B\beta_t^{1-\frac{b}{qa}} \left({\sum}_{k = 0}^{\infty} \mu_k \|\vPhi(\vx^k(\omega))\|^a\right)^\frac{b}{qa} + \|\bM_{u_t}(\omega)-\bM_{\ell_t}(\omega)\| \end{align*}
%\endgroup
%
\begingroup
\allowdisplaybreaks
\begin{align*} \|\vx^{u_t}(\omega)-\vx^{\ell_t}(\omega)\| & \leq B{\sum}_{k = \ell_t}^{u_t-1} \mu_k^{p_2}(1+\|\vPhi(\vx^k(\omega))\|^b)^\frac{1}{q} + \left\|{\sum}_{k = \ell_t}^{u_t-1} \mu_k^{p_1} \bA_k(\omega) \right\| \\ 
	&\leq  B{\sum}_{k = \ell_t}^{u_t-1} \mu_k   + B{\sum}_{k = \ell_t}^{u_t-1} \mu_k \|\vPhi(\vx^k(\omega))\|^{\frac{b}{q}}  + \left\|{\sum}_{k = \ell_t}^{u_t-1} \mu_k^{p_1} \bA_k(\omega) \right\| \\ 
	& \leq  B\beta_t + B\left({\sum}_{k = \ell_t}^{u_t-1} \mu_k\right)^{1-\frac{b}{qa}}\left({\sum}_{k = \ell_t}^{u_t-1} \mu_k \|\vPhi(\vx^k(\omega))\|^a\right)^\frac{b}{qa} \\ & \hspace{4ex} + \|\bM_{u_t}(\omega)-\bM_{\ell_t}(\omega)\| \\ & \leq B\beta_t + B\beta_t^{1-\frac{b}{qa}} \left({\sum}_{k = 0}^{\infty} \mu_k \|\vPhi(\vx^k(\omega))\|^a\right)^\frac{b}{qa} + \|\bM_{u_t}(\omega)-\bM_{\ell_t}(\omega)\| \end{align*}
\endgroup
for all $t$ sufficiently large, where we used the subadditivity of $x \mapsto x^\frac{1}{q}$,  $p_2 \geq 1$, and  $\mu_k \to 0$ in the second inequality, and H\"older's inequality in the third inequality. In the case $qa = b$, let us introduce $\bE_k(\omega) := \sum_{i=0}^{k-1}\mu_k\|\vPhi(\vx^k(\omega))\|^a$. Then, it follows
\[ \|\vx^{u_t}(\omega)-\vx^{\ell_t}(\omega)\| \leq B\beta_t + B[\bE_{u_t}(\omega)-\bE_{\ell_t}(\omega)] + \|\bM_{u_t}(\omega)-\bM_{\ell_t}(\omega)\| \]
for all $t$ sufficiently large. Due to $\beta_t \to 0$, $\bE_{k}(\omega) \to \bE(\omega) := \sum_{k=0}^\infty \mu_k\|\vPhi(\vx^k(\omega))\|^a$, and $\bM_k(\omega) \to \bM(\omega)$, we can infer $\|\vx^{u_t}(\omega)-\vx^{\ell_t}(\omega)\| \to 0$ in both cases. As before, the Lipschitz continuity of $\vPhi$ in \ref{P1} yields the contradiction
\[ \delta \leq \|\vPhi(\vx^{\ell_t}(\omega))\|  - \|\vPhi(\vx^{u_t}(\omega)) \| \leq \|\vPhi(\vx^{u_t}(\omega)) - \vPhi(\vx^{\ell_t}(\omega)) \|  \leq {\sf L}_{\Phi} \|\vx^{u_t}(\omega) - \vx^{\ell}(\omega)\| \to 0. \]
Thus, for all $\omega \in \mathcal M$, we have $\lim_{k \to \infty} \|\vPhi(\vx^k(\omega))\| = 0$. Since the event $\mathcal M$ occurs with probability $1$, this concludes the proof of item (ii).

\section{The supermartingale convergence theorem}

The following well-known and celebrated convergence theorem for supermartingale-type stochastic processes is due to Robbins and Siegmund \cite{RobSie71}.

\begin{thm}[Supermartingale Convergence Theorem]
	\label{Theorem:martingale_convergece}
	Let $\{y_k\}_{k\geq 0}$, $\{p_k\}_{k\geq 0}$, and $\{q_k\}_{k\geq 0}$ be sequences of nonnegative integrable random variables adapted to a filtration $\{\mathcal U_k\}_{k\geq 0}$. Furthermore, let $\{\beta_k\}_{k\geq 0} \subseteq \R_+$ be given with $\sum_{k=0}^\infty \beta_k < \infty$ and assume that we have 
	\[
	     \Exp[y_{k+1} \mid \mathcal U_k]\leq (1+\beta_k)y_k-p_k+ q_k
	\]
	 for all $k$ and $\sum_{k=0}^{\infty} q_k<\infty$ almost surely. Then, it holds that
	\begin{enumerate}[label=\textup{\textrm{(\alph*)}},topsep=0pt,itemsep=0ex,partopsep=0ex]
		\item If $\sum_{k=0}^\infty \Exp[q_k] < \infty$, then the sequence $\{\Exp[y_k]\}_{k\geq 0}$ converges to a finite number $\overline y$ and we have $\sum_{k=0}^{\infty} \Exp[p_k]<\infty$.
		\item $\{y_k\}_{k\geq 0}$ almost surely converges to a nonnegative finite random variable $y$ and it follows  $\sum_{k=0}^{\infty} p_k<\infty$ almost surely.
	\end{enumerate}
\end{thm}

\section{Proofs for Subsection 3.1}
%\section{Proofs for \Cref{sec:SGD}}
\subsection{Derivation of (6)}\label{appen:SGD}
%\subsection{Derivation of \eqref{eq:SGD recursion}}\label{appen:SGD}
Using the Lipschitz continuity of $\nabla f$ and the descent lemma, we obtain
\begingroup
\allowdisplaybreaks
\begin{align*}
	f(\vx^{k+1}) & \leq f(\vx^k) + \iprod{\nabla f(\vx^k)}{\vx^{k+1}-\vx^k} + \frac{{\sf L}}{2} \|\vx^{k+1}-\vx^k\|^2 \\ & = f(\vx^k) + \alpha_k \iprod{\nabla f(\vx^k)}{\nabla f(\vx^k)-\vg^k} - \alpha_k \|\nabla f(\vx^k)\|^2 + \frac{{\sf L}\alpha_k^2}{2} \|\vg^k\|^2 \\ & = f(\vx^k) - \alpha_k \left(1-\frac{{\sf L}\alpha_k}{2}\right) \|\nabla f(\vx^k)\|^2 + \alpha_k(1-{\sf L}\alpha_k) \iprod{\nabla f(\vx^k)}{\nabla f(\vx^k)-\vg^k} \\
	&\quad + \frac{{\sf L}\alpha_k^2}{2} \|\nabla f(\vx^k)-\vg^k\|^2. 
\end{align*}
\endgroup
Taking conditional expectation gives \eqref{eq:SGD recursion}.

\section{Results and proofs for Subsection 3.2}
%\section{Results and proofs for \Cref{sec:RR}}
We now provide more detailed algorithmic procedure and motivations for $\RR$. First, let us define the set of all possible permutations of $\{1,2,\ldots, N\}$  as 
\[
\Lambda:= \{\sigma: \sigma \text{ is a permutation of } \{1,2,\ldots, N\}\}. 
\]
At each iteration $k$, a  permutation $\sigma^{k+1}$   is generated according to an i.i.d. uniform distribution over $\Lambda$. Then, $\RR$ updates $\vx^{k}$ to $\vx^{k+1}$ through $N$ consecutive gradient descent-type steps by using the components $\{f(\cdot, \sigma^{k+1}_1), \ldots, f(\cdot, \sigma^{k+1}_N)\}$ sequentially, where  $\sigma^{k+1}_i$ represents the $i$-th element of $\sigma^{k+1}$.  In each  step, only one component $f(\cdot, \sigma^{k+1}_i)$ is selected for updating. To be more specific,  this method starts with $\tilde \vx^k_0 = \vx^{k}$ and then uses $f(\cdot, \sigma^{k+1}_i)$ to update $\tilde \vx^k_{i}$  as 
\[
\tilde \vx^k_{i} = \tilde \vx^k_{i-1} - \alpha_k \nabla f(\tilde \vx^k_{i-1}, \sigma^{k+1}_i) 
\]
for $i=1, 2, \ldots, N$, resulting in $\vx^{k+1} = \tilde \vx^k_N$. 

$\RR$ is used in a vast variety of engineering fields.  Most notably, $\RR$ is extensively applied in practice for training deep neural networks; see, e.g., \cite{gurbu2019,mishchenko2020,nguyen2020unified} and the references therein.

\subsection{A bound on the step length} \label{sec:app-RR-one}
We need the following lemma to for our later analysis. 
\begin{lemma}\label{lem:RR}
	We have the following estimate:
	\e\label{eq:RR step length}
	\|\vx^{k+1}-\vx^k\|^q \leq \bar {\sf A}{\sf G}(\vx^0)^{\frac{q}{2}} \alpha_k^q
	\ee
	for all $k$ sufficiently large and some positive constants $\bar {\sf A}$ and ${\sf G}(\vx^0)^{\frac{q}{2}}$.
\end{lemma}
\begin{proof}
	Using $|\sum_{i=1}^N a_i|^q \leq N^{q-1} \sum_{i=1}^N|a_i|^q$, \eqref{eq:RR}, and $\nabla f(\vx^k) = \frac{1}{N}\sum_{i=1}^N \nabla f(\vx^k,\sigma_i^{k+1})$, we have
	\begin{align*} \|\vx^{k+1}&-\vx^k\|^q = \alpha_k^q \left \|{\sum}_{i=1}^{N} \nabla f(\tilde \vx_{i-1}^k,\sigma_i^{k+1})\right\|^q \\
		& \leq 2^{q-1}N^q\alpha_k^q \|\nabla f(\vx^k)\|^q + 2^{q-1}\alpha_k^q \left \|{\sum}_{i=1}^{N} [\nabla f(\tilde \vx_{i-1}^k,\sigma_i^{k+1}) - \nabla f(\vx^k,\sigma_i^{k+1})]\right\|^q \\ & \leq 2^{q-1}N^q\alpha_k^q \|\nabla f(\vx^k)\|^q + (2N)^{q-1}{\sf L}^q \alpha_k^q \cdot {\sum}_{i=1}^N \|\tilde \vx_{i-1}^k - \vx^k\|^q. \end{align*}
	Setting $V_k := {\sum}_{i=1}^N \|\tilde \vx_{i-1}^k - \vx^k\|^q$, we recursively obtain
	\begin{align*}
		&V_k  = \alpha_k^q {\sum}_{i=2}^N \left\|{\sum}_{j=1}^{i-1} \nabla f(\tilde \vx_{j-1}^k,\sigma_j^{k+1})\right\|^q \\ & \leq 2^{q-1}\alpha_k^q {\sum}_{i=2}^N \left[ \left\|{\sum}_{j=1}^{i-1} [\nabla f(\tilde \vx_{j-1}^k,\sigma_j^{k+1})-\nabla f(\vx^k,\sigma_{j}^{k+1})]\right\|^q + \left\|{\sum}_{j=1}^{i-1} \nabla f(\vx^k,\sigma_j^{k+1})\right\|^q \right] \\ & \leq 2^{q-1}\alpha_k^q {\sum}_{i=2}^N \left[ {\sf L}^q(i-1)^{q-1} {\sum}_{j=1}^{i-1} \|\tilde \vx_{j-1}^k-\vx^k\|^q + \left[(i-1) {\sum}_{j=1}^{i-1} \|\nabla f(\vx^k,\sigma_j^{k+1})\|^2 \right]^{\frac{q}{2}} \right] \\ & \leq 2^{q-1}{\sf L}^q \alpha_k^q \left({\sum}_{i=1}^{N-1} i^{q-1} \right) V_k + 2^{\frac{3q}{2}-1}{\sf L}^{\frac{q}{2}} \alpha_k^q {\sum}_{i=2}^N (i-1)^{\frac{q}{2}}\left[{\sum}_{j=1}^{i-1} (f(\vx^k,\sigma_j^{k+1})-\bar f)\right]^{\frac{q}{2}} \\ & \leq 2^{q-1}{\sf L}^q \left[\frac{N^q}{q}\right] \alpha_k^q V_k + 2^{\frac{3q-2}{2}}{({\sf L}N)}^{\frac{q}{2}} \left[\frac{2N^{\frac{q}{2}+1}}{q+2}\right] \alpha_k^q (f(\vx^k)-\bar f)^{\frac{q}{2}}
	\end{align*}
	where we applied the estimate $\|\nabla f(\vx,i)\|^2 \leq 2{\sf L}(f(\vx,i)-\bar f)$ (see also \eqref{eq:nabla-f-psi} for comparison). Clearly, this establishes $V_k = \mathcal O(\alpha_k^q (f(\vx^k)-\bar f)^{\frac{q}{2}})$. Furthermore, following the proof of \cite[Lemma 3.2]{LiMiQiu21}, we have $f(\vx^k)-\bar f \leq {\sf G}(\vx^0)$ where ${\sf G}(\vx^0) = (f(\vx^0)-\bar f)\exp(2{\sf L}^3N^3\sum_{j=0}^\infty\alpha_j^3)$. Hence, using $\|\nabla f(\vx)\|^2 \leq 2{\sf L}(f(\vx)-\bar f)$ and \ref{B2}, there exists a constant $\bar{\sf A}$ such that \eqref{eq:RR step length} holds for all sufficiently large $k$. 
\end{proof}

\subsection{Proof of Corollary 3.2}\label{appen:RR}
%\subsection{Proof of \Cref{thm:RR}}\label{appen:RR}
We derive the convergence of $\RR$ below by setting $\vPhi \equiv \nabla f$ and $\mu_k \equiv \alpha_k$.

\textbf{Phase I: Verifying  \ref{P1}--\ref{P2}.}  \ref{B1} verifies condition \ref{P1} with ${\sf L}_\Phi = {\sf L}$. 
Towards verifying \ref{P2}, note that under the above assumptions for $\RR$, \cite[Lemma 3.1]{LiMiQiu21} establishes the recursion:
\e\label{eq:RR recursion}
f(\vx^{k+1}) - \bar f \leq (1+2{\sf L}^3N^3\alpha_k^3) [f(\vx^k)-\bar f] - \frac{N\alpha_k}{2} \|\nabla f(\vx^k)\|^2 - \frac{1-{\sf L}N\alpha_k}{2N\alpha_k} \|\vx^{k+1}-\vx^k\|^2
\ee
for all $k$ as long as $\alpha_k < 1/(\sqrt{2}{\sf L}N)$. (This always holds for large enough $k$ as $\alpha_k \to 0$). Taking total expectation and applying \Cref{Theorem:martingale_convergece} provides $\Exp[f(\vx^k) - \bar f] \to {\sf F}$ and $\sum_{k=0}^\infty \alpha_k \Exp[\|\nabla f(\vx^k)\|^2] < \infty$.  This verifies condition \ref{P2} with $a=2$.

\textbf{Phase II: Verifying  \ref{P3}--\ref{P4} for showing expected convergence.}
We can infer from \Cref{lem:RR} that \ref{P3} holds with arbitrary  $q\geq 1$ and $q\in \N$, ${\sf A} = \bar {\sf A}{\sf G}(\vx^0)^{\frac{q}{2}}$, $p_1 = q$, and ${\sf B} = 0$. \ref{P4} can be easily verified by these parameter choices and \ref{B2}. Thus, applying \Cref{thm:convergence theorem} gives $\Exp[\|\nabla f(\vx^k)\|] \to 0$. 

\textbf{Phase III: Verifying  \ref{P3'}--\ref{P4'} for showing almost sure convergence.}
According to the update \eqref{eq:RR},  we can let $\mA_k \equiv 0$, $p_2 = 1$,  $\bB_k = \sum_{i=1}^{N} \nabla f(\tilde \vx_{i-1}^k,\sigma_i^{k+1})$, $q = 2$, and $b=2$,  in \ref{P3'}. We have 
\[
\limsup_{k \to \infty} \frac{\|\bB_k\|^2} {1+\| \nabla f(\vx^k) \|^2} \leq \limsup_{k \to \infty} \frac{ 2\|\bB_k- {N}\nabla f(\vx^k)\|^2 + 2{N^2}\|\nabla f(\vx^k) \|^2} {1+\| \nabla f(\vx^k) \|^2}  < \infty,
\]
since $\|\bB_k- {N}\nabla f(\vx^k)\|^2 \leq \calO(\alpha_k^2)$  --- as shown in \Cref{sec:app-RR-one} --- which converges to $0$ as $k\to \infty$. Condition \ref{P3'} is verified. Assumption \ref{B2} and the previous parameter choices verify condition \ref{P4'}. Hence, we can apply \Cref{thm:convergence theorem} to derive $\|\nabla f(\vx^k)\| \to 0$ almost surely. 

Summarizing the above results yields \Cref{thm:RR}.

\section{Results and proofs for Subsection 3.3}
%\section{Results and proofs for \Cref{sec:prox-SGD}}
\subsection{Weakly convex functions}\label{appen:wcvx function}
A function $h : \Rn \to \Rex$ is said to be $\tau$-weakly convex if 
\[
h + \frac{\tau}{2}\|\cdot\|^2
\]
is convex (on its effective domain $\dom h = \{\vx \in \R^n: h(\vx) < \infty\}$). The class of weakly convex functions covers many important nonsmooth nonconvex problems. For instance, any function that has the composite form 
\[
h(\vx) = r(c(\vx))
\]
 with $r$ being convex Lipschitz continuous and  the Jacobian of $c$ being Lipschitz continuous is weakly convex. We refer to \cite{DavDru19} for more discussions.

If $h$ is $\tau$-weakly convex, proper, and lower semicontinuous, then the proximity operator $\prox_{\alpha h} : \Rn \to \Rn$, given by 
\[
\prox_{\alpha h}(\vx) := \argmin_{\vy \in \Rn}\,h(\vy) + \frac{1}{2\alpha}\|\vx-\vy\|^2,
\]
is a well-defined function for all $\alpha \in (0,\tau^{-1})$ and Lipschitz continuous with constant $(1-\alpha\tau)^{-1}$, see, e.g., \cite[Proposition 12.19]{RocWet98} or \cite[Proposition 3.3]{HohLabObe20}. The $\tau$-weak convexity is equivalent to 
\e \label{eq:def-weak-weak}
    h(\vy)\geq h(\vx)+\langle\vs,\vy-\vx\rangle-\frac{\tau}{2}\|\vx-\vy\|^2, \quad \forall~\vx, \vy, \  \ \forall~\vs\in \partial h(\vx)
\ee
see, e.g., \cite{Via83,DavDru19}. 

\subsection{Examples of weakly convex and Lipschitz continuous regularizers} \label{sec:some-examples}
In this section, we discuss several common nonconvex regularizers that can be shown to be weakly convex and Lipschitz continuous. The functions presented here have been mentioned in \cref{rem:examples}.

The minimax concave penalty (MCP), introduced in \cite{Zha10}, is the parametrized function $\varphi_{\lambda,\theta} : \R \to \R_+$ given by
\[ \varphi_{\lambda,\theta}(x) := \begin{cases} \lambda|x|-\frac{x^2}{2\theta} & \text{if } |x| \leq \theta\lambda, \\ \frac{\theta\lambda^2}{2} & \text{otherwise}, \end{cases} \]
where $\lambda, \theta > 0$ are two positive parameters. This function is $\theta^{-1}$-weakly convex and smooth for $x \neq 0$. Discussing the subdifferential $\partial \varphi_{\lambda,\theta}(x)$ and using \cite[Theorem 9.13]{RocWet98}, it can be shown that $\varphi_{\lambda,\theta}$ is Lipschitz continuous on $\R$ with modulus $\lambda$.

The smoothly clipped absolute deviation (SCAD) \cite{Fan97} is defined by
\[ \varphi_{\lambda,\theta}(x) := \begin{cases} \lambda|x|& \text{if } |x| \leq \lambda, \\ \frac{-x^2+2\theta\lambda|x|-\lambda^2}{2(\theta-1)} & \text{if } \lambda < |x| \leq \theta\lambda, \\ \frac{(\theta+1)\lambda^2}{2} & \text{if } |x| > \theta\lambda, \end{cases} \]
where $\lambda > 0$ and $\theta > 2$ are given parameters. The SCAD function is $(\theta-1)^{-1}$-weakly convex and Lipschitz continuous with modulus $\lambda$.

The student-t loss function is given by $\varphi_\theta(x) := \frac{\theta^2}{2}\log(1+\theta^{-2}x^2)$ for some $\theta \neq 0$. The first- and second-order derivative of $\varphi_\theta$ can be calculated as follows: 
\[ \varphi_\theta^\prime(x) = \frac{\theta^2x}{\theta^2+x^2} \quad \text{and} \quad \varphi_\theta^{\prime\prime}(x) = \frac{\theta^2(\theta^2-x^2)}{(\theta^2+x^2)^2}. \]
Some simple computations then show that $\varphi_\theta$ is $\frac18$-weakly convex and Lipschitz continuous with modulus ${|\theta|}/{2}$. Additional examples can be found in \cite{BoeWri21}.

\subsection{Equivalent stationarity measures}

We now first show that the two stationarity measures $\vx \mapsto \|F_{\mathrm{nat}}^1(\vx)\|$ and $\vx \mapsto \|\nabla \env_{\theta\psi}(\vx)\|$ are equivalent. This can be used to verify \cref{thm:prox-sgd}.

\begin{lemma} \label{lem:env-to-nat} Suppose that the conditions \ref{C1} and \ref{C2} are satisfied and let $\theta \in (0,[3{\sf L}+\tau]^{-1})$ and $\vx \in \Rn$ be given. Then, we have
	\[ (1-[3{\sf L}+\tau]\theta)\gamma\theta^{-2} \|F_{\mathrm{nat}}(\vx)\| \leq \|\nabla\env_{\theta\psi}(\vx)\| \leq {(1+[{\sf L}-\tau]\theta)}(\gamma+\tau){\theta^{-2}} \|F_{\mathrm{nat}}(\vx)\|, \]  
	where $\gamma = \theta/(1-[{\sf L}+\tau]\theta)$ and $F_{\mathrm{nat}} := F_{\mathrm{nat}}^1$.
\end{lemma}
\begin{proof} For every fixed $\vx \in \Rn$, let us define $\psi_{\vx}(\vy) := \psi(\vy) + \frac{{\sf L}+\tau}{2}\|\vx-\vy\|^2$ and $\varphi_{\vx}(\vy) := \varphi(\vy) + \frac{\tau}{2}\|\vx-\vy\|^2$. Then, setting $\gamma = \theta/(1-[{\sf L}+\tau]\theta)$, we may write
	\[ \prox_{\theta\psi}(\vx) = \prox_{\gamma\psi_{\vx}}(\vx) \quad \text{and} \quad \nabla \env_{\theta\psi}(\vx) = \theta^{-1}\gamma\nabla\env_{\gamma\psi_{\vx}}(\vx). \]
	Applying \cite[Theorem 3.5]{DruLew18} with $G \equiv \partial\varphi_{\vx}$, $\Phi \equiv \partial \psi_{\vx}$, $F \equiv \nabla f + {\sf L}(\cdot - x)$, $t \equiv \gamma$, and $\beta \equiv 2{\sf L}$, we can establish the following estimates for all $\vx \in \Rn$: 
	\[ {(1-2{\sf L}\gamma)\theta^{-1}} \|\vx-\prox_{\gamma\varphi_{\vx}}(\vx-\gamma F(\vx))\| \leq \|\nabla\env_{\theta\psi}(\vx)\|  \]  
and $\|\nabla\env_{\theta\psi}(\vx)\| \leq {(1+2{\sf L}\gamma)\theta^{-1}} \|\vx-\prox_{\gamma\varphi_{\vx}}(\vx-\gamma F(\vx))\|$. Moreover, it holds that
\begin{align*} \prox_{\gamma\varphi_{\vx}}(\vx-\gamma F(\vx)) & = \prox_{\gamma\varphi_{\vx}}(\vx-\gamma \nabla f(\vx)) \\
& = \argmin_{\vy \in \Rn}~\varphi(\vy) + \iprod{\nabla f(\vx)}{\vy-\vx} + \frac12 \left[\frac{1}{\gamma} + \tau \right] \|\vy - \vx\|^2 \\ & = \prox_{\frac{\gamma}{\gamma+\tau}\varphi}(\vx - \tfrac{\gamma}{\gamma+\tau}\nabla f(\vx)) \end{align*}
In addition, as shown in \cite[Lemma 2]{Nes13}, the functions $\delta \mapsto \|F^{1/\delta}_{\mathrm{nat}}(\vx)\|$ and $\delta \mapsto \|F^{1/\delta}_{\mathrm{nat}}(\vx)\|/\delta$ are decreasing and increasing in $\delta$, respectively. Hence, we can infer
\e \label{eq:so-many-lambdas} \min\{1,\lambda_2/\lambda_1\} \|F^{\lambda_2}_{\mathrm{nat}}(\vx)\| \leq \|F^{\lambda_1}_{\mathrm{nat}}(\vx)\| \leq \max\{1,\lambda_2/\lambda_1\} \|F^{\lambda_2}_{\mathrm{nat}}(\vx)\| \ee
for all $\lambda_1,\lambda_2 > 0$ and $\vx \in \Rn$. Choosing $\lambda_1 = {\gamma}/{(\gamma+\tau)}$ and $\lambda_2 = 1$ and noticing $1-2{\sf L}\gamma = (1-[3{\sf L}+\tau]\theta)\theta^{-1} > 0$, we can conclude the proof of \cref{lem:env-to-nat}.
\end{proof} 

\subsection{Proof of Lemma 3.4} \label{app:env}
%\subsection{Proof of \cref{lem:psgd-1}} \label{app:env}
Let us again note that the $\tau$-weak convexity of $\varphi$ --- as stated in \ref{C2} --- implies Lipschitz continuity of the proximity operator $\prox_{\theta\varphi}$ with modulus $(1-\theta\tau)^{-1}$. In the following results, we first bound the term $\|\vx^{k+1}-\bar \vx^k\|^2 = \|\vx^{k+1}-\prox_{\theta\psi}(\vx^k)\|^2$.

\begin{lemma} \label{lem:psi-to-phi} Suppose that \ref{C1} and \ref{C2} are satisfied and let $\theta \in (0,[{\sf L}+\tau]^{-1})$, $\alpha > 0$, $\vx \in \Rn$, and $\bar \vx = \prox_{\theta\psi}(\vx)$ be given. Then, it holds that
	\[ \bar \vx = \prox_{\alpha\varphi}(\bar \vx - \alpha \nabla f(\bar\vx) - \alpha\theta^{-1}[\bar \vx - \vx]). \]
\end{lemma}
\begin{proof} This is Lemma 3.2 in \cite{DavDru19}. 
\end{proof}

Our second result is analogous to Lemma 3.3 in \cite{DavDru19}. 

\begin{lemma} \label{lem:a-small-lemma} Suppose that the conditions \ref{C1}, \ref{C2}, and \ref{C4} are satisfied and let $\theta \in (0,[{\sf L}+\tau]^{-1})$ be given. Defining $\bar \vx^k = \prox_{\theta\psi}(\vx^k)$, $k \in \N$, it holds that
	\[ \Exp[\|\vx^{k+1} - \bar\vx^k\|^2 \mid \mathcal F_k] \leq (1-\alpha_k\theta_k)^2 \|\vx^k-\bar\vx^k\|^2 + \frac{\alpha_k^2\sigma_k^2}{(1-\alpha_k\tau)^2} \quad \text{almost surely} \quad \forall~k \in \N, \]
	where $\theta_k := (\theta^{-1}-[{\sf L}+\tau])/(1 - \alpha_k\tau)$ and $\sigma_k^2 := \Exp[\|\vg^k-\nabla f(\vx^k)\|^2 \mid \mathcal F_{k}]$. 
\end{lemma}
\begin{proof}
	%Using $\vx^{k+1} = \vx^k - \alpha_k[F_{\mathrm{nat}}^{\lambda_k}(\vx^k)-\ve^k]$, we have
	%
	%\begin{align*} \|\vx^{k+1} - \bar\vx^k\|^2 & =  \|\vx^k-\bar\vx^k\|^2 - 2\alpha_k \iprod{\vx^k-\bar\vx^k}{F_{\mathrm{nat}}^{\lambda_k}(\vx^k)-\ve^k} + \alpha_k^2 \|F_{\mathrm{nat}}^{\lambda_k}(\vx^k)-\ve^k\|^2  \\ 
	%	& = (1-\alpha_k) \|\vx^k-\bar\vx^k\|^2 - \alpha_k (1-\alpha_k) \|F_{\mathrm{nat}}^{\lambda_k}(\vx^k)-\ve^k\|^2 + \alpha_k \|\vx^k-\bar\vx^k-F_{\mathrm{nat}}^{\lambda_k}(\vx^k)+\ve^k\|^2, \end{align*}
	%
	%where we again applied $2\iprod{a}{b} = \|a\|^2 + \|b\|^2 - \|a-b\|^2$. Furthermore, 
	Invoking \cref{lem:psi-to-phi} and the Lipschitz continuity of $\prox_{\alpha_k\varphi}$ and $\nabla f$, it follows
	\begingroup
	\allowdisplaybreaks
	\begin{align*} 
		(1-\alpha_k\tau)^2\|\vx^{k+1}-\bar\vx^k\|^2 & \\ & \hspace{-23ex} = (1-\alpha_k\tau)^2 \|\prox_{\alpha_k\varphi}(\vx^k - \alpha_k \vg^k) - \prox_{\alpha_k\varphi}(\bar\vx^k - \alpha_k \nabla f(\bar\vx^k)-\alpha_k\theta^{-1}[\bar\vx^k-\vx^k])\|^2 \\ & \hspace{-19ex} \leq \|(1-\alpha_k\theta^{-1})[\bar\vx^k-\vx^k] + \alpha_k[\vg^k-\nabla f(\bar\vx^k)]\|^2 \\ & \hspace{-23ex} = (1-\alpha_k\theta^{-1})^2 \|\bar\vx^k-\vx^k\|^2 + 2\alpha_k(1-\alpha_k\theta^{-1}) \iprod{\bar\vx^k-\vx^k}{\vg^k-\nabla f(\bar\vx^k)} \\ & \hspace{-19ex} + \alpha_k^2 [ \|\vg^k-\nabla f(\vx^k)\|^2 + 2\iprod{\vg^k-\nabla f(\vx^k)}{\nabla f(\vx^k)-\nabla f(\bar\vx^k)} + \|\nabla f(\vx^k)-\nabla f(\bar\vx^k)\|^2] \\ & \hspace{-23ex} \leq (1-\alpha_k[\theta^{-1}-{\sf L}])^2 \|\bar\vx^k-\vx^k\|^2 + 2\alpha_k(1-\alpha_k\theta^{-1}) \iprod{\bar\vx^k-\vx^k}{\vg^k-\nabla f(\vx^k)} \\ & \hspace{-19ex}  + 2\alpha_k^2 \iprod{\vg^k-\nabla f(\vx^k)}{\nabla f(\vx^k)-\nabla f(\bar\vx^k)} + \alpha_k^2 \|\vg^k-\nabla f(\vx^k)\|^2. 
	\end{align*}
	\endgroup
	Taking conditional expectation, using \ref{C4} and $\vx^k,\bar\vx^k \in \mathcal F_k$, we obtain
	\[ \Exp[\|\vx^{k+1} - \bar\vx^k\|^2 \mid \mathcal F_k] \leq (1-\alpha_k\theta_k)^2 \|\vx^k-\bar\vx^k\|^2 + \frac{\alpha_k^2\sigma_k^2}{(1-\alpha_k\tau)^2}  \]
	almost surely, where $\theta_k = (\theta^{-1}-[{\sf L}+\tau])/(1-\alpha_k\tau)$.
	%where we applied Young's inequality---$\iprod{F_{\mathrm{nat}}^{\lambda_k}(\vx^k)}{\ve^k} \leq \frac14\|F_{\mathrm{nat}}^{\lambda_k}(\vx^k)\|^2 + \|\ve^k\|^2$---and the nonexpansiveness of the proximity operator. 
\end{proof}

By \ref{C4}, the stochastic error term $\sigma_k^2$ is bounded by ${\sf C}[f(\vx^k)-\bar f] + {\sf D}$ almost surely. In order to proceed, we need to link the function values ``$f(\vx^k)-\bar f$'' and ``$\env_{\theta\psi}(\vx^k)-\bar \psi$'' where $\bar\psi = \bar f+\bar\varphi$. The following lemma precisely establishes such a connection under assumption \ref{C3}.

\begin{lemma} \label{lem:psi-env} Suppose that the conditions \ref{C1}--\ref{C3} are satisfied and let $\theta \in (0,[\frac43{\sf L}+\tau]^{-1})$ be given. Then, it holds that
	\[ f(\vx) - \bar f \leq 2 [\env_{\theta\psi}(\vx)-\bar\psi] + {\sf L}_\varphi^2\theta \quad \forall~\vx \in \dom\varphi.\]
\end{lemma}

\begin{proof} Notice that the Lipschitz continuity of $\nabla f$ and assumption \ref{C2} imply $f(\vx - {\sf L}^{-1}\nabla f(\vx)) - f(\vx) \leq - \frac{1}{2{\sf L}} \|\nabla f(\vx)\|^2$ and
	\e \label{eq:nabla-f-psi} \|\nabla f(\vx)\|^2 \leq 2{\sf L}[f(\vx) - \bar f] = 2{\sf L}[\psi(\vx) - \varphi(\vx) - \bar f] \leq 2{\sf L}[\psi(\vx) - \bar \psi] \ee
	for all $\vx \in \dom\varphi$. Using \ref{C3}, the Lipschitz continuity of $\nabla f$, Young's inequality (twice), and $\theta < \frac34{\sf L}^{-1}$ and setting $\bar \vx = \prox_{\theta\psi}(\vx) \in \dom\varphi$, this yields
	\begingroup
	\allowdisplaybreaks
	\begin{align*}
		\psi(\vx) - \bar \psi & \leq \env_{\theta\psi}(\vx) - \bar\psi + \iprod{\nabla f(\bar \vx)}{\vx-\bar \vx} + \frac{1}{2}\left[{\sf L}-{\theta}^{-1}\right] \|\vx-\bar\vx\|^2 + \varphi(\vx)-\varphi(\bar\vx) \\ & \leq \env_{\theta\psi}(\vx) - \bar\psi + \frac{1}{2{\sf L}} \|\nabla f(\bar\vx)\|^2 + \left[ {\sf L} - \frac{1}{2\theta} \right] \|\vx-\bar\vx\|^2 + \frac{1}{4\theta} \|\vx-\bar\vx\|^2 + {\sf L}_{\varphi}^2\theta \\ & \leq  \env_{\theta\psi}(\vx) - \bar\psi  + [\psi(\bar\vx)-\bar \psi] + \frac{1}{2\theta} \|\vx - \bar \vx\|^2 + {\sf L}_\varphi^2\theta = \ 2[\env_{\theta\psi}(\vx) - \bar\psi] + {\sf L}_\varphi^2\theta. \end{align*}
	\endgroup
	This finishes the proof of \cref{lem:psi-env}. \end{proof} 

We now verify \cref{lem:psgd-1}. Choosing $\theta \in (0,[3{\sf L}+\tau]^{-1})$ ensures that \cref{lem:env-to-nat} and \cref{lem:psi-env} are applicable. Let $k \in \N$ be given with $\alpha_k \leq \min\{\frac{1}{2\tau},\frac{1}{2(\theta^{-1}-[{\sf L}+\tau])}\}$. This implies $1-\alpha_k\tau > \frac12$, 
%Moreover, due to \ref{C5}, we have $\alpha_k \to 0$ and hence, there exists $K \in \N$ such that  for all $k \geq K$. 
%
\[ 2{\sf L} \leq \theta^{-1}-[{\sf L}+\tau] \leq \theta_k \leq 2(\theta^{-1}-[{\sf L}+\tau]), \]
and $2-\theta_k\alpha_k \geq 1$. Utilizing the definition of the Moreau envelope, we can further infer
\begingroup
\allowdisplaybreaks
\begin{align} \nonumber \env_{\theta\psi}(\vx^{k+1}) & = \min_{\vy \in \Rn}~\psi(\vy) + \frac{1}{2\theta} \|\vy - \vx^{k+1}\|^2 \\ & \leq \env_{\theta\psi}(\vx^{k}) + \frac{1}{2\theta}\left[ \|\bar\vx^k - \vx^{k+1}\|^2 - \|\bar\vx^k - \vx^{k}\|^2 \right]. \label{eq:a-core-estimate} \end{align}
\endgroup
 Taking conditional expectation in \eqref{eq:a-core-estimate} and applying \cref{lem:a-small-lemma}, \ref{C4}, and \cref{lem:psi-env}, we obtain
\begin{align*} \Exp[\env_{\theta\psi}(\vx^{k+1})-\bar\psi \mid \mathcal F_k] & \leq [\env_{\theta\psi}(\vx^{k})-\bar\psi] + \frac{1}{2\theta}\left[ \Exp[\|\vx^{k+1} - \bar\vx^{k}\|^2 \mid\mathcal F_k] - \|\bar\vx^k - \vx^{k}\|^2 \right], \\ & \hspace{-10ex} \leq [\env_{\theta\psi}(\vx^{k})-\bar\psi] - \frac{2-\theta_k\alpha_k}{2\theta} \theta_k\alpha_k \|\bar \vx^k-\vx^k\|^2 + \frac{\alpha_k^2\sigma_k^2}{2\theta(1-\alpha_k\tau)^2} \\ & \hspace{-10ex} \leq \left[1 + \frac{4{\sf C}\alpha_k^2}{\theta}\right] [\env_{\theta\psi}(\vx^{k})-\bar\psi] - {\sf L}\theta \alpha_k \|\nabla\env_{\theta\psi}(\vx^k)\|^2 + 2\alpha_k^2 \left[ {\sf C}{\sf L}_\varphi^2 + \frac{{\sf D}}{\theta}\right],  \end{align*}
almost surely, where we used $\vx^k-\bar\vx^k = \theta \nabla \env_{\theta\psi}(\vx^k)$. This finishes the proof of \cref{lem:psgd-1}.

\subsection{Proof of Lemma 3.5} \label{sec:prox-sgd-expectation}
%\subsection{Proof of \cref{lem:psgd-2}} \label{sec:prox-sgd-expectation}
We first establish an upper bound for the natural residual that is analogous to \cite[Lemma A.1]{DavDruKakLee20}.

\begin{lemma} \label{lem:this-will-be-helpful} Suppose that the conditions \ref{C2} and \ref{C3} are satisfied. Let $\vx \in \dom\varphi$, $\vv \in \Rn$, and $\alpha \in (0,\frac12\tau^{-1})$ be given and set $F_v^\alpha(\vx) := \vx - \prox_{\alpha\varphi}(\vx-\alpha\vv)$. Then, it holds that
\e \|F^\alpha_v(\vx)\|^2 \leq 4{\sf L}_\varphi^2\alpha^2 + 4\|\vv\|^2\alpha^2. \label{eq:ok-ok} \ee
\end{lemma}

\begin{proof} 
Let us first note that if $h$ is $\tau$-weakly convex, proper, and lower semicontinuous, then the proximity operator $\prox_{\alpha h}$ can be equivalently characterized via the optimality condition:
\begin{equation} 
	\label{eq:let-it-be-prox} \vp = \prox_{\alpha h}(\vx) \quad \iff \quad \vx - \vp \in \alpha \partial h(\vp)
 \end{equation}
for all $\alpha \in (0,\tau^{-1})$. Consequently, since $\varphi$ is $\tau$-weakly convex, we have $\alpha^{-1}(F_v^\alpha(\vx)-\alpha \vv) \in \partial \varphi(\prox_{\alpha\varphi}(\vx-\alpha\vv))$. Using \eqref{eq:def-weak-weak}, this means
\[ \varphi(\vx) - \varphi(\prox_{\alpha\varphi}(\vx-\alpha\vv)) \geq \frac{1}{\alpha}\iprod{F_v^\alpha(\vx)-\alpha\vv}{F_v^\alpha(\vx)} - \frac{\tau}{2} \|F_v^\alpha(\vx)\|^2. \]
%
%see, e.g., \cite{Via83,DavDru19}. 
In addition, due to $\vx,\prox_{\alpha\varphi}(\vx-\alpha\vv) \in \dom\varphi$ and using Young's inequality, we have 
\[ \varphi(\vx)-\varphi(\prox_{\alpha\varphi}(\vx-\alpha\vv)) \leq {\sf L}_\varphi \|F_v^\alpha(\vx)\| \leq {\sf L}_\varphi^2\alpha + \frac{1}{4\alpha} \|F_v^\alpha(\vx)\|^2. \]
Using Young's inequality once more --- $\iprod{\vv}{F_v^\alpha(\vx)} \leq \alpha \|\vv\|^2 + \frac{1}{4\alpha}\|F_v^\alpha(\vx)\|$ --- it follows $\frac12[\frac{1}{\alpha}-\tau]\|F^\alpha_v(\vx)\|^2 \leq {\sf L}_\varphi^2\alpha + \alpha \|\vv\|^2$. The choice of $\alpha$ then readily implies \eqref{eq:ok-ok}.
\end{proof}

Based on \cref{lem:this-will-be-helpful}, we now verify \cref{lem:psgd-2}. Reusing the notation in \cref{lem:this-will-be-helpful}, we first observe
\[ \vx^{k+1} = \prox_{\alpha_k\varphi}(\vx^k-\alpha_k\vg^k) = \vx^k - F_{\vg^k}^{\alpha_k}(\vx^k). \]
Hence, by \eqref{eq:nabla-f-psi}, \cref{lem:psi-env}, and \ref{C4}, it follows
\begin{align*} \Exp[\|\vx^{k+1}-\vx^k\|^2 \mid \mathcal F_k] &= \Exp[\|F_{\vg^k}^{\alpha_k}(\vx^k)\|^2 \mid \mathcal F_k] \leq 4{\sf L}_\varphi^2\alpha_k^2 + 4\alpha_k^2 \Exp[\|\vg^k\|^2 \mid \mathcal F_k] \\ & =  4{\sf L}_\varphi^2\alpha_k^2 + 4\alpha_k^2 \|\nabla f(\vx^k)\|^2 +   4\alpha_k^2 \Exp[\|\nabla f(\vx^k)-\vg^k\|^2 \mid \mathcal F_k] \\ & \leq 4(2{\sf L}+{\sf C})\alpha_k^2 \cdot [f(\vx^k)-\bar f] +  4({\sf L}_\varphi^2+{\sf D})\alpha_k^2 \\ & \leq 8(2{\sf L}+{\sf C})\alpha_k^2 \cdot [\env_{\theta\psi}(\vx^k)-\bar\psi] +  4(((2{\sf L}+{\sf C})\theta+1){\sf L}_\varphi^2+{\sf D})\alpha_k^2. \end{align*}
This is exactly \eqref{eq:hello-hello} in \cref{lem:psgd-2}.

\subsection{Expected iteration complexity of prox-SGD}\label{appen:complexity PSGD}
As a byproduct of our analysis, we provide the expected iteration complexity of $\PSGD$ below. 
\begin{corollary}
Let $\{\vx^k\}_{k \geq 0}$ be generated by $\PSGD$ and let the assumptions \ref{C1}--\ref{C5} be satisfied. Then, for $\theta \in (0,[3{\sf L}+\tau]^{-1})$ and all $k$ with $\alpha_k \leq \min\{\frac{1}{2\tau},\frac{1}{2(\theta^{-1}-[{\sf L}+\tau])}\}$, it holds that 
\[
\min_{k=0,\ldots, T} \Exp[\|\nabla\env_{\theta\psi}(\vx^k)\|^2]  \leq \frac{\env_{\theta\psi}(\vx^{0})-\bar\psi + {\sf K} \sum_{k=0}^{T} \alpha_k^2}{{\sf L} \theta \sum_{k=0}^{T} \alpha_k}. 
\]
Here, ${\sf K}>0$ is defined in the proof. Consequently, if $\alpha_k = \frac{c}{(k+1) \log(k+2)}$ with some proper $c>0$, then we have 
\[
    \min_{k=0,\ldots, T} \Exp[\|\nabla\env_{\theta\psi}(\vx^k)\|^2]  \leq \calO\left(\frac{\log(T+2)}{\sqrt{T+1}}\right).
\]
\end{corollary}
\begin{proof}
	Taking total expectation in \Cref{lem:psgd-1} gives 
	\e\label{eq:complexity recursion 1}
	\begin{aligned}
		\Exp[\env_{\theta\psi}(\vx^{k+1}) - \bar\psi ] & \leq (1 + 4{\sf C}\theta^{-1}\alpha_k^2) \Exp[\env_{\theta\psi}(\vx^{k})-\bar\psi] \\ & \hspace{4ex} - {\sf L}\theta \alpha_k \Exp[\|\nabla\env_{\theta\psi}(\vx^k)\|^2] + 2\alpha_k^2 ( {\sf C}{\sf L}_\varphi^2 + {\sf D}{\theta^{-1}}).
	\end{aligned}
	\ee
	Then, by unrolling this recursion {and setting $c_1 := 4{\sf C}\theta^{-1}$ and $c_2 := 2({\sf C}{\sf L}_\varphi^2+{\sf D}\theta^{-1})$, we have 
	%\[
	\begin{align*}
		\Exp[\env_{\theta\psi}(\vx^{k+1}) - \bar\psi ]  & \leq \left\{ {\prod}_{j=0}^{k}(1 + c_1\alpha_j^2)\right\} \underbracket{[\env_{\theta\psi}(\vx^{0})-\bar\psi]}_{E_1}  \\ & \hspace{4ex} + c_2 {\sum}_{j=0}^{k-1}\left\{ {\prod}_{i=j+1}^{k}(1 + c_1\alpha_i^2) \right\} \alpha_j^2 + c_2 \alpha_k^2. \end{align*}
%		&=  E_1 \exp\left( {\sum}_{j=0}^{k}\log\left( 1+4{\sf C}\theta^{-1}\alpha_j^2\right)  \right) + E_2 \sum_{j=0}^{k}\alpha_j^2  \\
%		&\leq E_1 \exp\left( {\sum}_{j=0}^{k}4{\sf C}\theta^{-1}\alpha_j^2 \right) + E_2 \sum_{j=0}^{k}\alpha_j^2.
%	\end{aligned}
%	\]
	Next, using $\log(1+x) \leq x$ for all $x \geq 0$, we obtain the estimate 
	\[ {\prod}_{i=j}^{k}(1 + c_1\alpha_i^2) = \exp\left({\sum}_{i=j}^k \log(1+c_1\alpha_i^2)\right) \leq \exp\left({\sum}_{i=0}^\infty c_1\alpha_i^2\right) =: c_3 \]
	for all $j = 0,\dots,k$ and $k \in \N$. Thus, it follows $\Exp[\env_{\theta\psi}(\vx^{k+1}) - \bar\psi ] \leq c_3 E_1 + c_2c_3 \sum_{j=0}^k \alpha_j^2$ and we can infer $\Exp[\env_{\theta\psi}(\vx^{k}) - \bar\psi ] \leq {\sf G}$ for all $k$ and some constant $ {\sf G}>0$. Invoking this bound in \eqref{eq:complexity recursion 1} yields
\[		\Exp[\env_{\theta\psi}(\vx^{k+1}) - \bar\psi ] \leq \Exp[\env_{\theta\psi}(\vx^{k})-\bar\psi] - {\sf L}\theta \alpha_k \Exp[\|\nabla\env_{\theta\psi}(\vx^k)\|^2] + \alpha_k^2 \underbracket{(c_1{\sf G}+c_2)}_{\sf K}. \]
	}Finally, summing the above recursion from $k=0$ to $k =T$ and rearranging the terms, we have
	\[
	     \min_{k=0,\ldots, T} \Exp[\|\nabla\env_{\theta\psi}(\vx^k)\|^2]  \leq \frac{\Exp[\env_{\theta\psi}(\vx^{0})-\bar\psi] + {\sf K} \sum_{k=0}^{T} \alpha_k^2}{{\sf L} \theta \sum_{k=0}^{T} \alpha_k}. 
	\]
\end{proof}

\vspace{-0.5cm}
\section{Results and proofs for Subsection 3.4}
%\section{Results and proofs for \Cref{sec:SMM}}

\subsection{Three stochastic model-based methods}\label{appen:three SMM}
\vspace{-0.2cm}
Depending on the choice of the model function $f_{\vx^k}(\cdot, \xi^k)$, we recover different stochastic optimization methods. Specifically, we can consider the following examples: %we can have the following three different stochastic optimization methods:
\[
%\begin{cases}
\left[
\begin{aligned}
	f_{\vx^k}(\vx, \xi^{k})  & = f(\vx^k, \xi^{k}) + \iprod{\vs(\vx^{k}, \xi^{k})}{\vx - \vx^k)} \quad &\text{stochastic subgradient model,} \\
	f_{\vx^k}(\vx, \xi^{k})  & = f(\vx, \xi^{k}) \quad &\text{stochastic proximal point model,}
\end{aligned} \right.
%\end{cases}
\]
where $\vs(\vx^{k}, \xi^{k})\in \partial f(\vx^k,\xi^k)$ is a selected subgradient. In addition, if the weakly convex function $f$ has the composite form $f(\vx, \xi) = h(c(\vx,\xi),\xi)$, where $h(\cdot, \xi)$ is convex and Lipschitz continuous,  $c(\cdot, \xi)$ is smooth and has Lipschitz continuous Jacobian, we can further cover the prox-linear model:
\[
\begin{aligned}
	f_{\vx^k}(\vx, \xi^{k}) & = h(c(\vx^k,\xi^k) + \nabla c(\vx^k,\xi^k) (\vx- \vx^k), \xi^k), \quad & \text{stochastic prox-linear model.} \\
\end{aligned} 
\]

\subsection{Proof of Lemma 3.9}\label{appen:SMM step length}
%\Cref{lemma:step length model-based}}\label{appen:SMM step length}
\vspace{-0.2cm}
Under our assumptions, the estimate \cite[(4.8)]{DavDru19} is valid. {Furthermore, in \cite[Lemma 4.1]{DavDru19}, it is shown that the assumptions \ref{D1} and \ref{D3} imply Lipschitz continuity of the mapping $f$ (with constant ${\sf L}$). Hence, setting $\vx = \vx^k$ in \cite[(4.8)]{DavDru19},} we have 
\begingroup
\allowdisplaybreaks
\[
\begin{aligned}
	  \Exp[\|\vx^{k+1} - \vx^k\|^2 \mid \calF_k] &\leq -\frac{2\alpha_k}{1-\eta \alpha_k} \Exp[\psi(\vx^{k+1}) - \psi(\vx^k) \mid \calF_k] + \frac{2{\sf L}^2 \alpha_{k}^2}{1-\eta \alpha_k} \\
	  &\leq \frac{2({\sf L} + {\sf L}_\varphi )\alpha_k}{1-\eta \alpha_k} \Exp[\|\vx^{k+1} - \vx^k\| \mid \calF_k] + \frac{2{\sf L}^2 \alpha_{k}^2}{1-\eta \alpha_k} \\
	  &\leq \frac{1}{2} \Exp[\|\vx^{k+1} - \vx^k\|^2 \mid \calF_k] + \frac{2({\sf L} + {\sf L}_\varphi )^2\alpha_k^2}{(1-\eta \alpha_k)^2} + \frac{2{\sf L}^2 \alpha_{k}^2}{1-\eta \alpha_k}, 
\end{aligned}
\]
\endgroup
where we have used the Lipschitz continuity of $\psi$ in the second inequality, while the last inequality is due to Young's and Jensen inequalities. Rearranging terms gives 
\[
     \Exp[\|\vx^{k+1} - \vx^k\|^2 \mid \calF_k]  \leq \left(\frac{4({\sf L} + {\sf L}_\varphi )^2}{(1-\eta \alpha_k)^2} + \frac{4{\sf L}^2  }{1-\eta \alpha_k} \right) \alpha_k^2. 
\]
Thus, for all $k$ with $\alpha_k \leq 1/(2\eta)$, we obtain $\Exp[\|\vx^{k+1} - \vx^k\|^2 \mid \calF_k]  \leq (16({\sf L} + {\sf L}_\varphi )^2 + 8{\sf L}^2) \alpha_k^2$.
%Since $\alpha_k \to 0$ from \ref{D5}, there exists an integer $k_1$ such that $1-\eta \alpha_k \geq 1/2$ for all $k\geq k_1$. Thus, we have 
%\[
%\Exp[\|\vx^{k+1} - \vx^k\|^2 \mid \calF_k]  \leq \left(16({\sf L} + {\sf L}_\varphi )^2 + 8{\sf L}^2   \right) \alpha_k^2.
%\] 
%when $k \geq k_1$. 

{
\subsection{Convergence of {SMM}: An extended setting} \label{app:smm-extension}
In the following, we show that the analysis and results presented in the previous section and in \cref{sec:SMM} can be further strengthened and generalized. In particular, it is possible to work with assumptions that are more aligned with the conditions \ref{C1}--\ref{C4} for $\PSGD$. We consider the assumptions: 
\begin{enumerate}[label=\textup{\textrm{(E.\arabic*)}},topsep=0pt,itemsep=0ex,partopsep=0ex]
	\item \label{E1} The stochastic model function $f_{\vx}$ satisfies a one-sided accuracy property, i.e., we have $\Exp[f_{\vx^k}(\vx^k,\xi^k) \mid \mathcal F_k] = f(\vx^k)$ for all $k$ and 
	\[
	\Exp[f_{\vx^k}(\vy,\xi^k) - f(\vy) \mid \mathcal F_k] \leq \frac{\tau}{2} \|\vx^k-\vy\|^2 \quad \forall~\vy\in\dom \varphi, \quad \forall~k, \quad {\text{almost surely}}. \]
	\item \label{E2} The function $\vy \mapsto f_{\vx^k}(\vy,\xi^k)$ is $\eta$-weakly convex for all $k$.
	\item \label{E3} There exists a sequence $\{\vg^k\}_{k\geq 0}$ and constants ${\sf C}, {\sf D} \geq 0$ such that $\vg^k \in \partial f_{\vx^k}(\vx^k,\xi^k)$ and $\Exp[\|\vg^k\|^2 \mid \mathcal F_k] \leq {\sf C}[f(\vx^k)-\bar f] + {\sf D}$ almost surely for all $k$. 		
	\item \label{E4} The function $\varphi$ is $\rho$-weakly convex, proper, lower semicontinuous, and ${\sf L}_\varphi$-Lipschitz continuous on $\dom\varphi$. In addition, $\varphi$ is bounded from below on $\dom\varphi$, i.e., we have $\varphi(\vx) \geq \bar\varphi$ for all $\vx \in \dom\varphi$. 
	\item \label{E5} The step sizes $\{\alpha_k\}_{k \geq 0}$ satisfy $\sum_{k=0}^\infty \alpha_k = \infty$ and $\sum_{k=0}^\infty \alpha_k^2 < \infty$. 
\end{enumerate}
Concerning the regularity of $f$, we will work with \emph{one} of the following scenarios:  
\begin{enumerate}[label=\textup{\textrm{(F.\arabic*)}},topsep=0pt,itemsep=0ex,partopsep=0ex]
	\item \label{F1} The mapping $f$ is bounded from below and ${\sf L}_f$-Lipschitz continuous on $\dom \varphi$, i.e., there exists $\bar f$ such that $f(\vx) \geq \bar f$ for all $\vx \in \dom \varphi$ and we have $|f(\vx)-f(\vy)| \leq {\sf L}_f \|\vx - \vy\|$ for all $\vx,\vy \in \dom \varphi$.
	\item \label{F2} The function $f$ is bounded from below on $\Rn$, i.e., there is $\bar f$ such that $f(\vx) \geq \bar f$ for all $\vx \in \Rn$, and the gradient mapping $\nabla f$ is Lipschitz continuous (on $\Rn$) with modulus ${\sf L} > 0$. 
\end{enumerate}
In contrast to \cref{sec:SMM} and \cite{DavDru19}, we will not require explicit Lipschitz continuity of the model function $f_{\vx}$. Instead --- as we will verify now --- convergence of $\SMM$ can be established by only assuming Lipschitz continuity or Lipschitz smoothness of the mapping $f$. As mentioned in \Cref{appen:SMM step length}, the conditions \ref{D1} and \ref{D3} already imply Lipschitz continuity of $f$ (see again  \cite[Lemma 4.1]{DavDru19}). Hence, {assuming \ref{F1} can be more general.} Moreover, the alternative assumption \ref{F2} allows to cover the case where $f$ is not Lipschitz continuous but sufficiently smooth. This situation appears more frequently in stochastic proximal point methods where the model function is chosen as $f_{\vx^k}(\vx,\xi^k) = f(\vx,\xi^k)$. Assumption \ref{E2} {is parallel to assumption} \ref{D2}. Assumption \ref{E3} is a mild variance-type condition that is similar to \ref{C4} and which can be weaker than \ref{D3}. We now establish the core estimates provided in \cref{lemma:recursion model-based} and \cref{lemma:step length model-based} under the more general assumptions \ref{E1}--\ref{E5} and \ref{F1} or \ref{F2}. Our analysis is inspired by the proof of \cite[Lemma 4.2]{DavDru19} but is closer to \cite[Appendix C]{MilSchUlb22}.

\textit{Lipschitz continuous $f$}. We first investigate core properties of $\SMM$ under \ref{F1}. Let us define $\Psi_k(\vx) := f_{\vx^k}(\vx,\xi^k)+\varphi(\vx)$ and $F_k(\vx) := f_{\vx^k}(\vx,\xi^k) + \frac{\eta}{2}\|\vx-\vx^k\|^2$. Then, by assumption, the mapping $\vx \mapsto \Psi_k(\vx) + \frac{1}{2{\alpha_k}}\|\vx-\vx^k\|^2$ is strongly convex with parameter ${\alpha_k^{-1}}-\zeta$, where $\zeta = \eta+\rho$. Specifically, if $\alpha_k < \zeta^{-1}$ and by \eqref{eq:model based}, we have
\e \label{eq:yes-its-extra} \Psi_k(\vx) + \frac{1}{2\alpha_k}\|\vx-\vx^k\|^2 \geq \Psi_k(\vx^{k+1}) + \frac{1}{2\alpha_k}\|\vx^{k+1}-\vx^k\|^2 + \frac12[\alpha_k^{-1}-\zeta] \|\vx - \vx^{k+1}\|^2 \ee
for all $\vx \in \dom \varphi$ and due to the convexity of $F_k$ and the Lipschitz continuity of $f$, it holds that
\begin{align} \nonumber \Psi_k(\vx) - \Psi_k(\vx^{k+1}) & = [\psi(\vx)-\psi(\vx^{k+1})] + [f_{\vx^k}(\vx,\xi^k)-f(\vx)] + [f(\vx^{k+1})-f(\vx^k)] \\ \nonumber & \hspace{4ex} + [f(\vx^k)-f_{\vx^k}(\vx^k,\xi^k)] + [F_k(\vx^k)-F_k(\vx^{k+1})] + \frac{\eta}{2}\|\vx^{k+1}-\vx^k\|^2 \\ \nonumber & \leq  [\psi(\vx)-\psi(\vx^{k+1})] + [f_{\vx^k}(\vx,\xi^k)-f(\vx)] + [f(\vx^k)-f_{\vx^k}(\vx^k,\xi^k)]  \\ & \hspace{4ex} + \iprod{\vg^k}{\vx^k-\vx^{k+1}} + {\sf L}_f \|\vx^{k+1}-\vx^k\| + \frac{\eta}{2}\|\vx^{k+1}-\vx^k\|^2, \label{eq:intermediate-psi} \end{align}
where we used $\vg^k \in \partial f_{\vx^k}(\vx^k,\xi^k) = \partial F_k(\vx^k)$. Upon taking conditional expectation and using the Cauchy-Schwarz and Young's inequality and \ref{E1}, we obtain
\begin{align*} \Exp[\Psi_k(\vx) - \Psi_k(\vx^{k+1}) \mid \mathcal F_k] & \leq \Exp[\psi(\vx)-\psi(\vx^{k+1}) \mid \mathcal F_k] + 2\alpha_k \Exp[\|\vg^k\|^2 \mid \mathcal F_k] + 2{\sf L}_f^2\alpha_k \\ & \hspace{4ex} + \frac{\tau}{2} \|\vx^k-\vx\|^2 + \frac12 \left[\eta+\frac{1}{{2}\alpha_k}\right] \Exp[\|\vx^{k+1}-\vx^{k}\|^2 \mid \mathcal F_k]. \end{align*}
Rearranging the terms in \eqref{eq:yes-its-extra}, this yields
\begin{align} \nonumber  \Exp[\|\vx^{k+1}-\vx\|^2 \mid \mathcal F_k] & \\ \nonumber & \hspace{-12ex} \leq \frac{1+\alpha_k\tau}{1-\alpha_k\zeta} \|\vx^k-\vx\|^2 + \frac{2\alpha_k}{1-\alpha_k\zeta} \Exp[\psi(\vx)-\psi(\vx^{k+1}) \mid \mathcal F_k] \\ & \hspace{-8ex} + \frac{4\alpha_k^2}{1-\alpha_k\zeta} \Exp[\|\vg^k\|^2 \mid \mathcal F_k] + \frac{4{\sf L}_f^2\alpha_k^2}{1-\alpha_k\zeta} - \frac{1-2\alpha_k\eta}{2(1-\alpha_k\zeta)} \Exp[\|\vx^{k+1}-\vx^{k}\|^2 \mid \mathcal F_k]. \label{eq:quick-new-result} \end{align}
Let us now define $\bar \vx^k := \prox_{\theta\psi}(\vx^k)$ for $\theta \in (0,\zeta^{-1})$. Then, due to $\vx^k,\bar\vx^k \in \dom \varphi$ and {the Lipschitz continuity of $f$ and $\varphi$} and applying Young's inequality, it holds that
\e\label{eq:bound f}
\begin{aligned}
	f(\vx^k) - \bar f \leq \psi(\vx^k)-\bar\psi & \leq \env_{\theta\psi}(\vx^k) + ({\sf L}_f+{\sf L}_\varphi) \|\vx^k-\bar\vx^k\| - \frac{1}{2\theta} \|\vx^k-\bar\vx^k\|^2 - \bar \psi \\ & \leq \env_{\theta\psi}(\vx^k) - \bar\psi + 0.5\theta{({\sf L}_f+{\sf L}_\varphi)^2}.  
\end{aligned}
\ee
In addition, setting $\vx = \bar \vx^k$ in \eqref{eq:quick-new-result}, using
\[ \psi(\bar \vx^k) = \env_{\theta\psi}(\vx^k) - \frac{1}{2\theta}\|\vx^k-\bar\vx^k\|^2 \leq \psi(\vx^{k+1}) + \frac{1}{2\theta} [\|\vx^{k+1}-\vx^k\|^2 - \|\vx^k-\bar\vx^k\|^2] \]
and combining \ref{E3} and \eqref{eq:bound f}, we obtain
\begin{align}
	\nonumber \Exp[\|\vx^{k+1}-\bar\vx^k\|^2 \mid \mathcal F_k] & \leq \left[1-\frac{(\theta^{-1}-\zeta-\tau)\alpha_k}{1-\alpha_k\zeta}\right] \|\vx^k-\bar\vx^k\|^2 + \frac{4{\sf C}\alpha_k^2}{1-\alpha_k\zeta} [\env_{\theta\psi}(\vx^{k})-\bar\psi] \\ & \hspace{4ex} + \frac{{\sf G}\alpha_k^2}{1-\alpha_k\zeta} - \frac{1-2[\eta+\theta^{-1}]\alpha_k}{2(1-\alpha_k\zeta)} \Exp[\|\vx^{k+1}-\vx^{k}\|^2 \mid \mathcal F_k], \label{eq:bound step length}
\end{align}
where ${\sf G} := 2{\sf C}\theta({\sf L}_f+{\sf L}_\varphi)^2 + 4({\sf L}_f^2+{\sf D})$. Hence, {using the definition of the Moreau envelope and \eqref{eq:bound step length}}, for $\theta \in (0,(\zeta+\tau)^{-1})$ and all $k$ with $\alpha_k \leq \min\{\frac{1}{2\zeta},\frac{\theta}{2(1+\theta\eta)}\}$, it follows
\begingroup
\allowdisplaybreaks
\begin{align} \nonumber \Exp[\env_{\theta\psi}(\vx^{k+1})-\bar\psi \mid \mathcal F_k] & \\ \nonumber & \hspace{-16ex} \leq [\env_{\theta\psi}(\vx^{k})-\bar\psi] + \frac{1}{2\theta}\left[ \Exp[\|\vx^{k+1} - \bar\vx^{k}\|^2 \mid\mathcal F_k] - \|\bar\vx^k - \vx^{k}\|^2 \right], \\ %& \hspace{-10ex} \leq [\env_{\theta\psi}(\vx^{k})-\bar\psi] - \frac{2-\theta_k\alpha_k}{2\theta} \theta_k\alpha_k \|\bar \vx^k-\vx^k\|^2 + \frac{\alpha_k^2\sigma_k^2}{2\theta(1-\alpha_k\tau)^2} \\ & \hspace{-10ex}
	& \hspace{-16ex} \leq \left[1 + \frac{4{\sf C}\alpha_k^2}{\theta}\right] [\env_{\theta\psi}(\vx^{k})-\bar\psi] - \frac{1-\theta(\zeta+\tau)}{2} \alpha_k \|\nabla\env_{\theta\psi}(\vx^k)\|^2 + \frac{{\sf G}\alpha_k^2}{\theta}. \label{eq:nice-extra} \end{align}
\endgroup
{This gives an estimate similar to \Cref{lemma:recursion model-based}.}

Next, let us consider an iteration $k$ with $\alpha_k \leq \min\{\frac{1}{2\zeta},\frac{1}{2\eta}\}$. Using \eqref{eq:quick-new-result} with $\vx = \vx^k$ and Young's inequality, we obtain
\begin{align*} \frac12 \Exp[\|\vx^{k+1}-\vx^k\|^2 \mid \mathcal F_k] & \\ & \hspace{-16ex} \leq {2\alpha_k}\Exp[\psi(\vx^{k})-\psi(\vx^{k+1}) \mid \mathcal F_k] + {4\alpha_k^2} \Exp[\|\vg^k\|^2 \mid \mathcal F_k] + {4{\sf L}_f^2\alpha_k^2} \\ & \hspace{-16ex} \leq 2({\sf L}_f+{\sf L}_\varphi)\alpha_k\Exp[\|\vx^{k+1}-\vx^k\| \mid \mathcal F_k] + {4{\sf C}\alpha_k^2}[\env_{\theta\psi}(\vx^k)-\bar\psi] + {\sf G}\alpha_k^2 \\ & \hspace{-16ex} \leq \frac14 \Exp[\|\vx^{k+1}-\vx^k\|^2 \mid \mathcal F_k] + {4{\sf C}\alpha_k^2}[\env_{\theta\psi}(\vx^k)-\bar\psi] + [{\sf G}+4({\sf L}_f+{\sf L}_\varphi)^2]\alpha_k^2. \end{align*}
Rearranging the terms yields $\Exp[\|\vx^{k+1}-\vx^k\|^2 \mid \mathcal F_k] \leq {16{\sf C}\alpha_k^2}[\env_{\theta\psi}(\vx^k)-\bar\psi] + 4[{\sf G}+4({\sf L}_f+{\sf L}_\varphi)^2]\alpha_k^2$, {which is similar to the estimates in \Cref{lemma:step length model-based} (indeed, it is more similar to \Cref{lem:psgd-2})}.

With these estimates and {by following exactly the derivations in \Cref{sec:prox-SGD}}, we can  show $\Exp[\|\nabla \env_{\theta\psi}(\vx^k)\|] \to 0$ and $\|\nabla \env_{\theta\psi}(\vx^k)\| \to 0$ almost surely. 

\textit{Lipschitz smooth $f$}. We continue our discussion under \ref{F2}. Using the Lipschitz continuity of $\nabla f$, the estimate \eqref{eq:intermediate-psi} changes to
\begin{align*} \nonumber \Psi_k(\vx) - \Psi_k(\vx^{k+1}) & \leq  [\psi(\vx)-\psi(\vx^{k+1})] + [f_{\vx^k}(\vx,\xi^k)-f(\vx)] + [f(\vx^k)-f_{\vx^k}(\vx^k,\xi^k)]  \\ & \hspace{4ex} + \iprod{\vg^k-\nabla f(\vx^k)}{\vx^k-\vx^{k+1}} + \frac{{\sf L}+\eta}{2}\|\vx^{k+1}-\vx^k\|^2.  \end{align*}
Taking conditional expectation and applying Young's inequality and \ref{E1}, this yields
\begin{align*} \Exp[\Psi_k(\vx) - \Psi_k(\vx^{k+1}) \mid \mathcal F_k] & \leq \Exp[\psi(\vx)-\psi(\vx^{k+1}) \mid \mathcal F_k] + 2\alpha_k \Exp[\|\vg^k\|^2 \mid \mathcal F_k] + 2\alpha_k\|\nabla f(\vx^k)\|^2 \\ & \hspace{4ex} + \frac{\tau}{2} \|\vx^k-\vx\|^2 + \frac12 \left[{\sf L}+\eta+\frac{1}{2\alpha_k}\right] \Exp[\|\vx^{k+1}-\vx^{k}\|^2 \mid \mathcal F_k]. \end{align*}
By \cref{lem:psi-env} and \eqref{eq:nabla-f-psi}, we can further infer $f(\vx) - \bar f \leq 2 [\env_{\theta\psi}(\vx)-\bar\psi] + {\sf L}_\varphi^2\theta$ and $\|\nabla f(\vx)\|^2 \leq 2{\sf L}[f(\vx)-\bar f]$ for all $\vx \in \dom \varphi$. (Notice that the conditions \ref{C1}--\ref{C3} and \ref{E4} and \ref{F2} are identical). Hence, similar to \eqref{eq:quick-new-result}, it follows
\begin{align} \nonumber  \Exp[\|\vx^{k+1}-\vx\|^2 \mid \mathcal F_k] & \\ \nonumber & \hspace{-12ex} \leq \frac{1+\alpha_k\tau}{1-\alpha_k\zeta} \|\vx^k-\vx\|^2 + \frac{2\alpha_k}{1-\alpha_k\zeta} \Exp[\psi(\vx)-\psi(\vx^{k+1}) \mid \mathcal F_k] + \frac{\tilde{\sf G}\alpha_k^2}{1-\alpha_k\zeta} \\ & \hspace{-8ex} + \frac{8({\sf C}+2{\sf L})\alpha_k^2}{1-\alpha_k\zeta} [\env_{\theta\psi}(\vx^k)-\bar\psi]  - \frac{1-2\alpha_k({\sf L}+\eta)}{2(1-\alpha_k\zeta)} \Exp[\|\vx^{k+1}-\vx^{k}\|^2 \mid \mathcal F_k], \label{eq:even-newer-result} \end{align}
where $\tilde{\sf G} = 4(({\sf C}+2{\sf L}){\sf L}_\varphi^2 + {\sf D})$ and $\zeta = \eta+\rho$. At this point, we can fully mimic our earlier calculations. In particular, similar to \eqref{eq:nice-extra}, for $\theta \in (0,(\zeta+\tau)^{-1})$ and all $k$ with $\alpha_k \leq \min\{\frac{1}{2\zeta},\frac{\theta}{2(1+\theta({\sf L}+\eta))}\}$, we obtain
\begin{align*} \nonumber \Exp[\env_{\theta\psi}(\vx^{k+1})-\bar\psi \mid \mathcal F_k] &\leq \left[1 + \frac{8{(\sf C}+2{\sf L})\alpha_k^2}{\theta}\right] [\env_{\theta\psi}(\vx^{k})-\bar\psi] \\
	&\hspace{4ex} - \frac{1-\theta(\zeta+\tau)}{2} \alpha_k \|\nabla\env_{\theta\psi}(\vx^k)\|^2 + \frac{\tilde{\sf G}\alpha_k^2}{\theta}.  \end{align*}
Setting $\vx = \vx^k$ in \eqref{eq:even-newer-result} and using the Lipschitz continuity of $\varphi$ and $\nabla f$ and Young's inequality, it holds that 
\begingroup
\allowdisplaybreaks
\begin{align*} &\frac12 \Exp[\|\vx^{k+1}-\vx^k\|^2 \mid \mathcal F_k]  \leq {2\alpha_k}\Exp[\psi(\vx^{k})-\psi(\vx^{k+1}) \mid \mathcal F_k] \\& \hspace{25ex} + {8({\sf C}+2{\sf L})\alpha_k^2}[\env_{\theta\psi}(\vx^k)-\bar\psi] + {\tilde{\sf G}\alpha_k^2} \\ & \leq 2{\sf L}_\varphi \alpha_k\Exp[\|\vx^{k+1}-\vx^k\| \mid \mathcal F_k] + {\sf L}\alpha_k \Exp[\|\vx^{k+1}-\vx^k\|^2 \mid \mathcal F_k] \\ & \hspace{4ex} -2\alpha_k \Exp[\iprod{\nabla f(\vx^k)}{\vx^{k+1}-\vx^k} \mid \mathcal F_k] + {8({\sf C}+2{\sf L})\alpha_k^2}[\env_{\theta\psi}(\vx^k)-\bar\psi]  + \tilde {\sf G}\alpha_k^2 \\ & \leq \frac{ \Exp[\|\vx^{k+1}-\vx^k\|^2 \mid \mathcal F_k]}{4}  + 16\alpha_k^2 \|\nabla f(\vx^k)\|^2    + {8({\sf C}+2{\sf L})\alpha_k^2}[\env_{\theta\psi}(\vx^k)-\bar\psi] + [\tilde {\sf G}+16{\sf L}_\varphi^2]\alpha_k^2 \\ & \leq  \frac14 \Exp[\|\vx^{k+1}-\vx^k\|^2 \mid \mathcal F_k] +   {8({\sf C}+10{\sf L})\alpha_k^2}[\env_{\theta\psi}(\vx^k)-\bar\psi]  + [\tilde {\sf G}+16(1+2{\sf L}\theta){\sf L}_\varphi^2]\alpha_k^2 \end{align*}
\endgroup
for all $k$ with $\alpha_k \leq \min\{\frac{1}{2\zeta},\frac{1}{8{\sf L}},\frac{1}{2({\sf L}+\eta)}\}$. As before, this establishes variants of \cref{lemma:recursion model-based} and \cref{lemma:step length model-based} {(more precisely, \Cref{lem:psgd-2})} and allows to follow the derivations in \Cref{sec:prox-SGD} to establish convergence results. 

Finally, we summarize all the above observations in the following corollary. 

\begin{corollary}\label{thm:model based new}
	Let us consider the family of stochastic model-based methods \eqref{eq:model based} for the problem \eqref{eq:wcvx prob} under assumptions \ref{E1}--\ref{E5} and \ref{F1} or \ref{F2}. Then, {for all $\theta \in (0,(\eta + \rho +\tau)^{-1})$}, we have $\lim_{k \to \infty} \Exp[\|\nabla \env_{\theta\psi}(\vx^k)\| ] = 0$ and $\lim_{k \to \infty} \|\nabla \env_{\theta\psi}(\vx^k)\| = 0$ almost surely. 
\end{corollary}
}

{
\section{Comparison: related literature}\label{appen:table}

\begin{table*}[!htp]\caption{Summary and comparison of related and representative literature.}\label{table:literature}
	\begin{center}
			\small{
			\begin{tabular}{cccc}
				\specialrule{1pt}{0pt}{0pt} \\[-1.5ex]
				 \multirow{2}*{$\SGD$} & \multirow{2}*{Assumptions} & \multicolumn{2}{c}{Convergence} \\ \cmidrule(lr){3-4}
				  & & in expectation & almost surely \\[0.5ex]
				\hline \\[-1.5ex]
				\cite{BerTsi00} &\btab{c}  \ref{A1}, \ref{A2}, \ref{A4},  \\  bounded variance (\ref{A3} with ${\sf C} = 0$)\etab &  \xmark &  \cmark  \\[0.5ex]
				\hline \\[-1.5ex]
				\cite{mertikopoulos2020}&  \btab{c}  \ref{A1}, \ref{A4} \\  $f$ is coercive  ($\implies$ \ref{A2}) \\ $f$ is Lipschitz \\ $\liminf_{\|x\|\to \infty} \|\nabla f(x)\| >0$  \\  bounded variance (\ref{A3} with ${\sf C} = 0$) \etab
				& \xmark &\cmark 
				\\[0.5ex]
				\hline \\[-1.5ex]
			    \cite{BotCurNoc18}&  \btab{c}  \ref{A1}, \ref{A2},  \ref{A4}  \\ $f$ is twice differentiable \\ $x\mapsto \nabla^2 f(x) \nabla f(x)$ is Lipschitz \\  bounded variance (\ref{A3} with ${\sf C} = 0$) \etab
				& \cmark &\xmark 
				\\[0.5ex]
				\hline \\[-1.5ex]
				This work&  \ref{A1}--\ref{A4}
				& \cmark &\cmark 
				\\[0.5ex]
				\specialrule{1pt}{0pt}{0pt} \\[-0.5ex]
				\specialrule{1pt}{0pt}{0pt} \\[-1.5ex]
				 \multirow{2}*{$\RR$} & \multirow{2}*{Assumptions} & \multicolumn{2}{c}{Convergence} \\ \cmidrule(lr){3-4}
				  & & in expectation & almost surely \\[0.5ex]
				\hline \\[-1.5ex]
				\cite{LiMiQiu21} & \ref{B1}--\ref{B2} &  \xmark &  \cmark \\[0.5ex]
				\hline \\[-1.5ex]
				This work & \ref{B1}--\ref{B2} &  \cmark &  \cmark  \\[0.5ex]
				\specialrule{1pt}{0pt}{0pt} \\[-0.5ex]
				\specialrule{1pt}{0pt}{0pt} \\[-1.5ex]
				\multirow{2}*{$\PSGD$} & \multirow{2}*{Assumptions} & \multicolumn{2}{c}{Convergence} \\ \cmidrule(lr){3-4}
				  & & in expectation & almost surely \\[0.5ex]
				\hline \\[-1.5ex]
				\cite{MajMiaMou18} & \btab{c} \ref{C1}, \ref{C3}, \ref{C5} \\ %$\varphi$ is convex ($\implies$ \ref{C2}) \\ 
				$\{\vx^k\}_{k\geq 0}$ is surely bounded \\ almost surely bounded noise  \etab&  \xmark &  \cmark  \\[0.5ex]
				\hline \\[-1.5ex]
				This work & \ref{C1}--\ref{C5} &  \cmark &  \cmark  \\[0.5ex]
				\specialrule{1pt}{0pt}{0pt} \\[-0.5ex]
				\specialrule{1pt}{0pt}{0pt} \\[-1.5ex]
				 \multirow{2}*{$\SMM$} & \multirow{2}*{Assumptions} & \multicolumn{2}{c}{Convergence} \\ \cmidrule(lr){3-4}
				  & & in expectation & almost surely \\[0.5ex]
				\hline \\[-1.5ex]
				\cite{DucRua18} & \btab{c} \ref{D3}, \ref{D5} \\ surely one-sided accuracy ($\implies$ \ref{D1}) \\ $\varphi$ is convex ($\implies$ \ref{D2})) \\  $\{\vx^k\}_{k\geq 0}$ is surely bounded (compact constraint) \\  density / Sard-type condition \etab&  \xmark &  \cmark  \\[0.5ex]
				\hline \\[-1.5ex]
				This work & \ref{D1}--\ref{D5} &  \cmark &  \cmark  \\[0.5ex]
				\specialrule{1pt}{0pt}{0pt} \\[-1.5ex]
			\end{tabular}
			}
	\end{center}
\end{table*}

}

\vspace{-0.4cm}
{
\section{Non-asymptotic complexity vs. asymptotic convergence}\label{appen:motivations}
\vspace{-0.2cm}
In this subsection, we provide several additional arguments and illustrations that can help to explain and illuminate the potential differences between typical finite-step complexity rates and asymptotic convergence results --- as obtained in \Cref{thm:convergence theorem}. To this end, we consider the standard optimization problem 
\e \label{eq:app-smooth} \min_{\vx \in \Rn} f(\vx), \ee 
where $f : \R^n \to \R$ is a given, smooth, and nonconvex function. As motivated in the introduction, complexity bounds for the nonconvex problem \eqref{eq:app-smooth} typically take the form      
%
%uppose we apply SGD to minimize a smooth nonconvex function f . SGD generates a
%25 sequence of iterates {xk}k≥0, which is a stochastic process due to the randomness of the algorithm
%26 and the utilized stochastic oracles. The most commonly seen ‘convergence result’ for SGD is the
%27 expected iteration complexity, which typically takes the form [15]
%
%28 where T denotes the total number of iterations and k ̄ is an index sampled uniformly at random
%29 from {0, . . . , T }. Note that we ignored some higher-order convergence terms and constants to
%30 ease the presentation. Complexity results are integral to understand core properties and progress
%31 of the algorithm during the first T iterations, while the asymptotic convergence behavior plays an
%32 equally important role as it characterizes whether an algorithm can eventually approach an exact
%33 stationary point or not;
%
\e \label{eq:app-bounds} \min_{k = 0,\dots,T} \Exp[\|\nabla f(\vx^k)\|^2] = \mathcal O(({T+1})^{-\frac12}) \quad \text{or} \quad \Exp[\|\nabla f(\vx^{\bar k})\|^2] \leq \mathcal O(({T+1})^{-\frac12}).  \ee
Here, $T$ denotes the total number of iterations, the index $\bar k$ is sampled uniformly at random from $\{0,\dots,T\}$, and the iterates $\{\vx^k\}_{k \geq 0}$ are assumed to be generated by the stochastic gradient descent method, see \cite{ghadimi2013,BotCurNoc18,lei2019stochastic,khaled2020better}. Similar (deterministic) complexity results are also available for the basic gradient descent method, \cite{Beck2017}, and many other related algorithmic schemes, \cite{GhaLanZha16,DavDru19,mishchenko2020,nguyen2020unified}. In particular, for the gradient descent method, the complexity bounds \eqref{eq:app-bounds} can be strengthened to 
\e \label{eq:app-bounds-det} \min_{k = 0,\dots,T} \|\nabla f(\vx^k)\|^2 = \mathcal O(({T+1})^{-1}),  \ee
see, e.g., \cite[Theorem 10.15]{Beck2017}.

While the complexity bounds shown in \eqref{eq:app-bounds} (and \eqref{eq:app-bounds-det}) allow to capture and characterize the overall trend of the minimization procedure during the first $T$ iterations, they cannot fully justify a common practice in stochastic optimization: the last iterate $\vx^T$ is returned as final output of the algorithm. In fact, even as $T$ increases, the term $\Exp[\|\nabla f(\vx^T)\|]$ can be arbitrarily large whereas the complexity measure $\min_{k = 0,\dots,T} \Exp[\|\nabla f(\vx^k)\|^2]$ decreases at its respective rate. The asymptotic convergence results 
\[ \lim_{k \to \infty} \Exp[\|\nabla f(\vx^k)\|] = 0 \quad \text{or} \quad \lim_{k \to \infty} \|\nabla f(\vx^k)\| = 0 \quad \text{almost surely}, \]
can provide additional information: as $\{\Exp[\|\nabla f(\vx^k)\|]\}_{k\geq 0}$ converges to zero, the term $\Exp[\|\nabla f(\vx^{T})\|]$ will stay small (below any predefined threshold) for all $T$ sufficiently large. In tandem with the complexity bounds \eqref{eq:app-bounds}, this supports common output strategies that return the last iterate $\vx^T$ --- at least for large $T$. Hence, both non-asymptotic and asymptotic convergence analyses are important and informative --- especially in the nonconvex and stochastic setting --- %. \Cref{thm:convergence theorem} provides a general and simple plugin-type tool allowing to convert non-asymptotic convergence bounds to asymptotic convergence guarantees.
and the combination of non-asymptotic and asymptotic convergence guarantees can paint a more complete picture of the convergence behavior of stochastic optimization methods. 

We continue with a specific example that can illustrate the mentioned discrepancies and differences between non-asymptotic complexity and asymptotic convergence results. In particular, we construct a nonconvex function $f : \R \to \R$, a corresponding step size sequence $\{\alpha_k\}_{k\geq 0}$, and an initial point $\vx^0 \in \R$, for which the standard gradient descent method generates a sequence of iterates $\{\vx^k\}_{k\geq 0}$ that satisfies the non-asymptotic complexity bound \eqref{eq:app-bounds-det}, but we can \textit{not} observe asymptotic convergence $f^\prime(\vx^k) \to 0$. 

\begin{figure}[t]
	\setlength{\abovecaptionskip}{-3pt plus 3pt minus 0pt}
	\setlength{\belowcaptionskip}{-10pt plus 3pt minus 0pt}
	\centering
	\subfigure[]{\includegraphics[width=6.7cm]{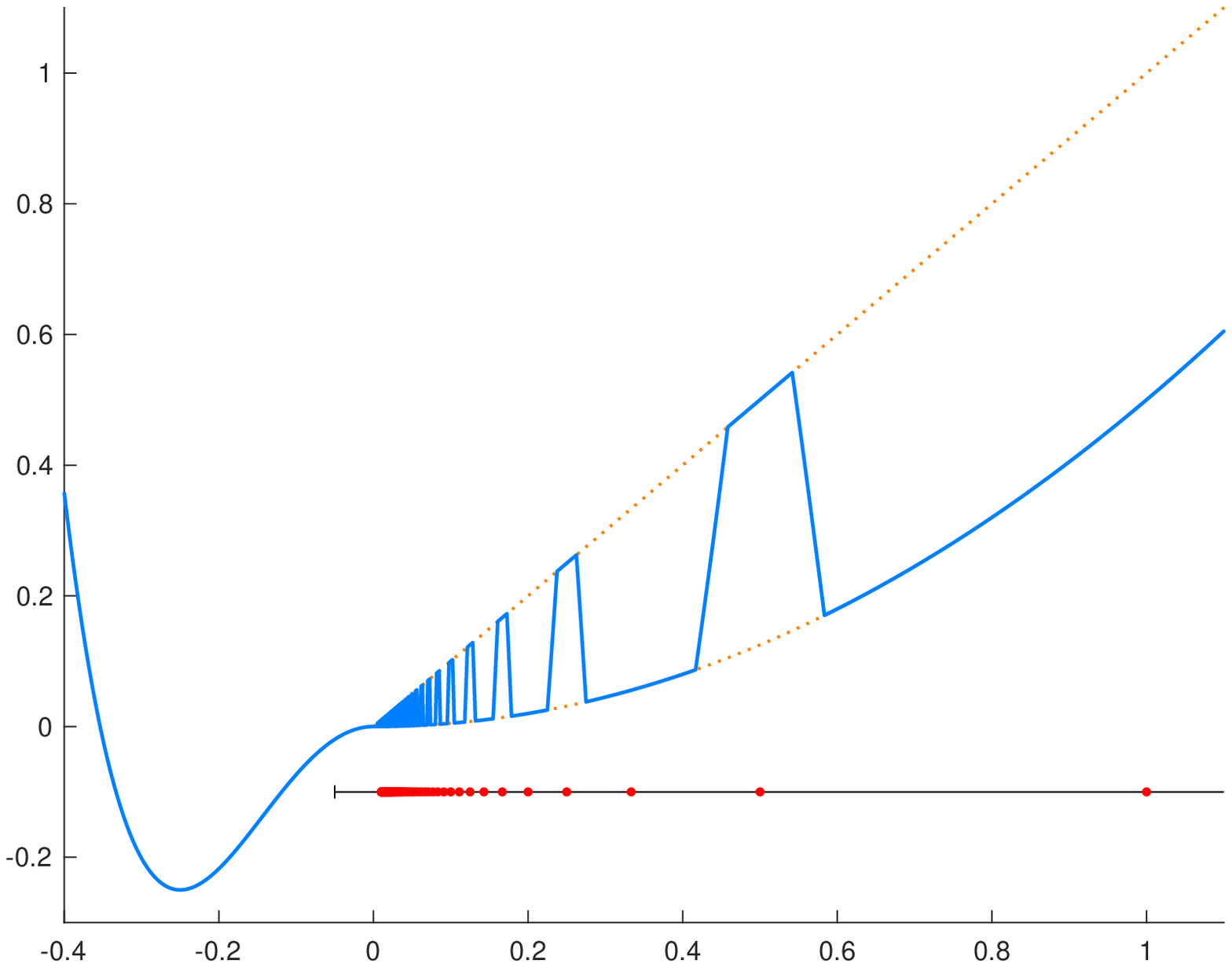}} \hspace{2.5ex}
	\subfigure[]{\includegraphics[width=6.7cm]{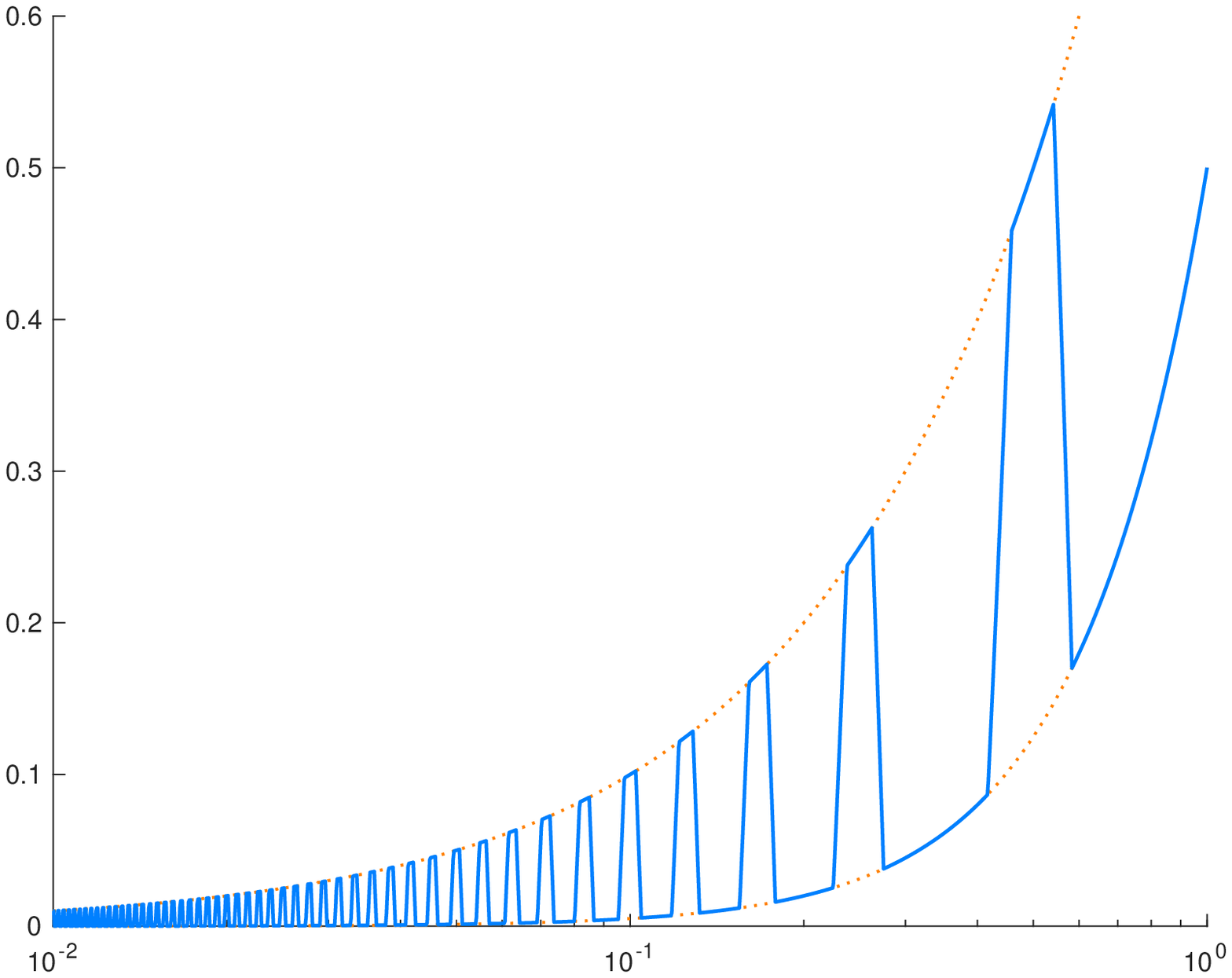}}
	\caption{Plot of $f$. The function $f$ is depicted using a blue color. The dotted orange lines correspond to the functions $x \mapsto x$ and $x \mapsto \frac12 x^2$, respectively. The read points in subfigure (a) show the iterates $\vx^k$, $k \in \N$. Subfigure (b) shows a logarithmic plot of $f$ on $[0.01,1]$.}
	\label{fig:example}
\end{figure}

Let $0 \leq \nu < \kappa$ be given parameters. We consider the functions
\[ c(x) := \begin{cases} e^{-\frac1x} & \text{if } x > 0, \\ 0 & \text{otherwise}, \end{cases} \quad \text{and} \quad \bar c_{\kappa,\nu}(x) := \frac{c(\kappa^2-x)}{c(\kappa^2-x)+c(x-\nu^2)}. \]
The mappings $c$ and $\bar c_{\kappa,\nu}$ are obviously $C^\infty$ and we have $\bar c_{\kappa,\nu}(x) = 0$ for all $x \geq \kappa^2$ and $\bar c_{\kappa,\nu}(x) = 1$ for all $x \leq \nu^2$. Moreover, it holds that 
\[ c^\prime(x) = \begin{cases} \frac{1}{x^2} e^{-\frac1x} & \text{if } x > 0, \\ 0 & \text{otherwise}, \end{cases} \quad \text{and} \quad \bar c^\prime_{\kappa,\nu}(x) = - \frac{c^\prime(\kappa^2-x)c(x-\nu^2)+c(\kappa^2-x)c^\prime(x-\nu^2)}{(c(\kappa^2-x)+c(x-\nu^2))^2}. \]
We now define $\kappa_j := \frac{1}{4j(2j+1)}$, $\nu_j := \frac{1}{8j(2j+1)}$, $g(x) := x-\frac12 x^2$, $\gamma_j(x) := g(x) \bar c_{\kappa_j,\nu_j}((x-\frac{1}{2j})^2)$, and
\[ f(x) := \begin{cases} h(x) + \sum_{k=1}^\infty \gamma_k(x) & \text{if } x \neq 0, \\ 0 & \text{if } x = 0, \end{cases}   \quad \text{and} \quad h(x) := \begin{cases} \frac12 x^2 & \text{if } x \geq 0, \\ 8x^2(8x^2-1) & \text{if } x < 0. \end{cases} \]
An exemplary plot of the function $f$ is shown in \Cref{fig:example}. The function $f$ is continuous on $\R$ and smooth for all $x \neq 0$. We now want to run gradient descent on $f$ with initial point $\vx^0 = 1$ and diminishing step sizes
\[ \alpha_k := \frac{1}{k+2}, \;\; \text{if $k$ is even} \quad \text{and} \quad  \alpha_k := \frac{1}{(k+1)(k+2)}, \;\; \text{if $k$ is odd.} \] %\alpha_k := \begin{cases} \frac{1}{k+2} & \text{if $k$ is even}, \\ \frac{1}{(k+1)(k+2)} & \text{if $k$ is odd}. \end{cases} \]
Then, it follows
\begin{equation} \label{eq:app-example} \vx^k = \frac{1}{k+1} \quad \text{and} \quad f^\prime(\vx^k) = \begin{cases} 1 & \text{if $k$ is even}, \\ \frac{1}{k+1} & \text{if $k$ is odd}. \end{cases} \end{equation}
We now verify this statement by induction. We first notice that the functions $\{\gamma_k\}_{k\geq 0}$ and derivatives $\{\gamma_k^\prime\}_{k\geq0}$ have disjoint supports $[\frac{1}{2k}-\frac{1}{4k(2k+1)},\frac{1}{2k}+\frac{1}{4k(2k+1)}] = [\frac{1}{2}(\frac{1}{2k}+\frac{1}{2k+1}),\frac{1}{2}(\frac{3}{2k}-\frac{1}{2k+1})]$ with center point $\frac{1}{2k}$, $k \in \N$. More specifically, we have
\[ \gamma_k(x) = \begin{cases} 0 & \text{if } |x-\frac{1}{2k}| \geq \frac{1}{4k(2k+1)}, \\ g(x) & \text{if } |x-\frac{1}{2k}| \leq \frac{1}{8k(2k+1)},  \end{cases} \quad \text{and} \quad \gamma_k^\prime (x) = \begin{cases} 0 & \text{if } |x-\frac{1}{2k}| \geq \frac{1}{4k(2k+1)}, \\ g^\prime(x) & \text{if } |x-\frac{1}{2k}| \leq \frac{1}{8k(2k+1)}.  \end{cases} \] 
Hence, \eqref{eq:app-example} clearly holds for the base case $k = 0$. Let us now assume that the induction hypothesis \eqref{eq:app-example} is true for some $k$ and let $k+1 = 2j$, $j \in \mathbb N$, be an even number. Then, we obtain
\begin{align*} \vx^{k+1} & = \vx^k - \alpha_k f^\prime(\vx^k) \\ &= \frac{1}{k+1} - \frac{1}{(k+1)(k+2)} \cdot \left[\frac{1}{k+1} + \gamma_j^\prime\left(\frac{1}{k+1}\right)\right] \\ & =  \frac{1}{k+1} - \frac{1}{(k+1)(k+2)} \cdot 1 = \frac{1}{k+2}. \end{align*}
Similarly, if $k+1 = 2j+1$, $j \in \mathbb N$, is odd, we then have $\frac{1}{2j}-\frac{1}{k+1} = \frac{1}{2j(2j+1)} > \frac{1}{4j(2j+1)}$ and $\frac{1}{k+1}-\frac{1}{2j+2} = \frac{1}{(2j+1)(2j+2)} > \frac{1}{(2j+2)(2j+3)}$. This yields
\[ \vx^{k+1} = \vx^k - \alpha_k f^\prime(\vx^k) = \frac{1}{k+1} - \frac{1}{k+2}\cdot \frac{1}{k+1} = \frac{1}{k+2}. \]
This finishes the proof of \eqref{eq:app-example}. We further notice %$\sum_{k=0}^\infty \alpha_k \geq \sum_{k=0}^\infty \alpha_{2k} = \sum_{k=0}^\infty \frac{1}{2(k+1)} = \infty$ and $\sum_{k=0}^\infty \alpha_k^2 < \infty$. 
\[ \sum_{k=0}^\infty \alpha_k \geq \sum_{k=0}^\infty \alpha_{2k} = \sum_{k=0}^\infty \frac{1}{2(k+1)} = \infty \quad \text{and} \quad \sum_{k=0}^\infty \alpha_k^2 < \infty. \]
Consequently, the step sizes $\{\alpha_k\}_k$ satisfy all standard requirements. In addition, we have 
\[ \sum_{k=0}^\infty \alpha_k |f^\prime(\vx^k)|^2 \leq \sum_{k=0}^\infty \frac{1}{(k+1)(k+2)} < \infty \quad \text{and} \quad \min_{k=0,\dots,T} |f^\prime(\vx^k)|^2 \leq \frac{1}{T^2}. \] 
Thus, the complexity results in \eqref{eq:app-bounds-det} obviously hold, but the gradient values $f^\prime(\vx^k)$ do not converge to zero. In particular, for even $T$, the last iterate $\vx^T$ satisfies $f^\prime(\vx^T) = 1$ which obstructs interpretability of $\vx^T$ and of the complexity bounds $\min_{k=0,\dots,T} |f^\prime(\vx^k)|^2 \leq \varepsilon$ or $|f^\prime(\vx^{\bar k})|^2 \leq \varepsilon$. Notice that the mapping $f$ is not Lipschitz smooth around $x = 0$ and hence, the convergence results in \Cref{thm:convergence theorem} are not applicable here. Of course, these observations have even higher significance in the stochastic setting, when evaluation of the bounds \eqref{eq:app-bounds} is generally restrictive or not possible within the algorithmic procedure.  

%These difficulties are not fully captured by the complexity results.
%

We conclude and summarize our discussion with a comment by Francesco Orabona on non-asymptotic and asymptotic convergence analyses for $\SGD$ (\cite{Orabona20}, blog post: ``\textit{Almost sure convergence of SGD on smooth non-convex functions}'', section 5 and 6, Oct. 05, 2020): 

\begin{center}
\hspace{-2ex}\begin{minipage}{0.95\linewidth}
\begin{tikzpicture}
      \node[fill=ivory!25, draw=ivory!50, rectangle, rounded corners, minimum width=3.5cm, minimum height=1.2cm, align=center,   
        inner sep=5pt] {\begin{varwidth}{\linewidth}\textit{``Note that the 20-30 years ago there were many papers studying the asymptotic convergence of SGD and its variants in various settings. Then, the taste of the community changed moving from asymptotic convergence to finite-time rates. As it often happens when a new trend takes over the previous one, new generations tend to be oblivious to the old results and proof techniques. The common motivation to ignore these past results is that the finite-time analysis is superior to the asymptotic one, but this is clearly false (ask a statistician!). It should be instead clear to anyone that both analyses have pros and cons.''}\end{varwidth}};
\end{tikzpicture}
\end{minipage}
\end{center}

}

\end{document}